\documentclass[12pt,leqno,a4paper]{amsart}
\usepackage{amssymb}  
\usepackage{amsmath}  
\usepackage{amsthm}  
\usepackage{verbatim} 

\theoremstyle{plain}
\newtheorem{theorem}{THEOREM}[section]
\newtheorem{proposition}{PROPOSITION}[section]
\newtheorem{lemma}{LEMMA}[section]
\newtheorem{corollary}{COROLLARY}[section]

\theoremstyle{definition}
\newtheorem{definition}{DEFINITION}[section]

\newtheorem{remark}{Remark}[section]

\numberwithin{equation}{section}
\newcommand{\ga}{\alpha}

\newcommand{\gd}{\delta}
\newcommand{\gD}{\Delta}

\newcommand{\gl}{\lambda}
\newcommand{\gL}{\Lambda}
\newcommand{\gm}{\mu}
\newcommand{\gn}{\nu}

\newcommand{\gp}{\pi}

\newcommand{\gs}{\sigma}

\newcommand{\go}{\omega}
\newcommand{\gO}{\Omega}
\newcommand{\ov}{\overline}
\newcommand{\de}{\partial}

\newcommand{\R}{\mathbb {R}}
\newcommand{\RN}{\R ^N}
\newcommand{\Rm}{\R ^m}

\newcommand{\ldue}{\textbf{L}^2}
\newcommand{\huno}{\textbf{H}_0^1}

\newcommand{\be}{\begin{equation}}
\newcommand{\ee}{\end{equation}}

\begin{document}

\title[Symmetry for cooperative elliptic systems]
{Symmetry results for cooperative elliptic systems via linearization}
\author[Damascelli]{Lucio Damascelli}
\address{ Dipartimento di Matematica, Universit\`a  di Roma 
" Tor Vergata " - Via della Ricerca Scientifica 1 - 00173  Roma - Italy.} 
\email{damascel@mat.uniroma2.it}
\author[Pacella]{Filomena Pacella}
\address{Dipartimento di Matematica, Universit\`a di Roma
" La Sapienza " -  P.le A. Moro 2 - 00185 Roma - Italy.}
\email{pacella@mat.uniroma1.it}
\date{}
\thanks{Supported by PRIN-2009-WRJ3W7 grant}
\begin{abstract} In this paper we prove symmetry results for classical solutions of nonlinear cooperative elliptic systems in a ball or in annulus in $\RN$, $N \geq 2 $.
More precisely we prove that solutions having Morse index $j \leq N $ are foliated Schwarz symmetric if the nonlinearity is convex and a full coupling condition is satisfied along the solution.
   \end{abstract}

\maketitle



\section{  \textbf{Introduction and statement of the results} }
\label{se:1}

We consider a semilinear elliptic system of the type
\begin{equation} \label{modprob} 
\begin{cases} - \gD U =F(|x|,U) \quad &\text{in } \gO \\
U= 0 \quad & \text{on } \de  \gO
\end{cases}
\end{equation}
where $F=(f_1 , \dots , f_m)$ is a function belonging to $C^{1, \ga }([0,+\infty )\times \Rm; \Rm)$ and $\gO$ is a bounded domain in $\RN$, $m, N \geq 2$.
Here $U=(u_1, \dots ,u_m)$ is a vector valued function in $\gO$.\\
It is well known that systems of this type arise in many applications in different fields, we refer to \cite{Me}, \cite{Mu} for some of these.\par
We are interested in studying symmetry properties of classical solutions of \eqref{modprob} when $\gO$ is  rotationally symmetric i.e. a ball or an annulus centered at the origin  in $\RN$.  If $\gO$ is a ball and the system is cooperative  then by using the famous ''moving plane method'' (\cite{S}, \cite{GNN}) it can be proved 
(see \cite{BS1}, \cite{deF1}, \cite{deF2}, \cite{Tr} ) that every positive solution $U$ (i.e. $U_i >0$ in $\gO$ for any $i=1,\dots ,m$) is radial and radially decreasing under the additional hypothesis that each $f_i$ is nonincreasing with respect to $|x|$.\\
Here we consider solutions of any sign, do not require $f$ to be monotone in $|x|$ and would like to obtain results also in the case of an annulus in the same direction of the results obtained in \cite{P} and \cite{PW} for scalar equations. We recall that in \cite{P} and \cite{PW} (see also \cite{GPW}) it was proved that classical solutions with Morse index less than or equal to $N$ are foliated Schwarz symmetric if the nonlinearity is convex or its first derivative is convex and $\gO$ is a ball or an annulus.\\
Another symmetry result obtained by symmetrization for some particular systems, is contained in \cite{KP}.
\par 
Let us give the definition of a foliated Schwarz symmetric function.
\begin{definition}\label{foliatedSS}
Let $\gO$ be a rotationally symmetric domain in $\RN$, $N\geq 2$. We say that a continuous vector valued function $U=(u_1 , \dots ,u_m): \gO \to \Rm$ is foliated Schwarz symmetric if each component $u_i$ is foliated Schwarz symmetric with respect to the same vector $p \in \RN$. In other words there exists a vector $p \in \RN$, 
$|p|=1$, such that $U(x)$ depends only on $r=|x|$ and $\theta= \arccos \left ( \frac {x}{|x|}\cdot p  \right ) $ and $U$ is (componentwise) nonincreasing in $\theta $.
\end{definition}
\begin{remark} Let us observe that if $U$ is a solution of \eqref{modprob} and the system satisfies some coupling conditions, as required in Theorem \ref{fconvessa}, then the foliated Schwarz symmetry of $U$ implies that either $U$ is radial or it is strictly decreasing in the angular variable $\theta $. This will be deduced by the proof of Theorem \ref{fconvessa}.
\end{remark}
In order to state our results we need to define the Morse index for solutions of \eqref{modprob} and the coupling conditions for the system.  
However they will be restated and commented in Section 2 for general semilinear elliptic systems.
\begin{definition}\label{MorseIndex}
\begin{itemize}
\item [i)]
Let $U$ be a $C^2 (\gO;\Rm)$ solution of \eqref{modprob}. We say that $U$ is linearized stable (or that it has zero Morse index) if the quadratic form 
\begin{equation} 
\begin{split}
Q_U (\Psi ;\gO  ) &= \int _{\gO} \left [ |\nabla \Psi |^2 - J_F (|x|, U)(\Psi ,\Psi ) \right ]dx= \\
& \int _{\gO}\left [   \sum _{i=1}^m |\nabla \psi _i |^2 -\sum _{i,j=1}^m \frac {\de f_i}{\de u_j}(|x|, U(x)) \psi _i \psi _j \right ]  \, dx \geq 0
\end{split}
\end{equation}
for any $\Psi = (\psi _1, \dots , \psi _m) \in C_c^1 (\gO;\Rm)$ where $J_F (x,U)$ is the jacobian matrix of $F$ computed at $U$.
\item [ii)] $U$ has (linearized) Morse index  equal to the integer $\mu=\mu (U)\geq 1$ if $\mu $ is the maximal dimension of a subspace of  $ C_c^1 (\gO;\Rm)$ where the quadratic form is negative definite.
\end{itemize}
\end{definition}

\begin{definition} 
\begin{itemize}
\item We say that the system \eqref{modprob} is cooperative or weakly coupled in an open set $\gO' \subseteq \gO $  if 
$$  \frac {\de f _i} {\de u _j} (|x|, u_1, \dots , u_m) \geq 0 \quad \forall \; (x,u_1,\dots ,u_m ) \in \gO '  \times  \Rm 
$$
for any  $i,j=1,\dots ,m$ with  $ i \neq j$.
  \item  We say that the system \eqref{modprob} is fully  coupled along a solution $U$ in an open set $\gO' \subseteq \gO $ if  it is cooperative in $\gO ' $ and 
 in addition $\forall I,J \subset \{1, \dots ,m \}$ such that $I \neq \emptyset $, $J \neq \emptyset $, $I \cap J = \emptyset $, $I \cup J = \{1, \dots ,m \} $ there exist $i_0 \in I$, $j_0 \in J $ such that 
$$\text{meas }(\{ x \in \gO ' :   \frac {\de f _{i_0}} {\de u _{j_0}} (|x|, U(x)) > 0 \}) >0$$  
 \end{itemize}
\end{definition}
As it will be clear in the following sections the previous definition reflects the fact that the linearized system at a solution $U$ is weakly or fully coupled, according to Definition \ref{sistemilineariaccoppiati}.\par
Let us point out that  the last definition is slightly different from the usual one stated in the literature, because the ''full coupling conditions'' on the derivatives are required only when computed on the solution but on a subset $\gO '$. It is not difficult to see that there are systems which are fully coupled along some solution but not fully coupled in the usual sense. This was observed and used in \cite{MMP} and we will give some examples in Section 4.  \par 
\smallskip
Let $e\in S^{N-1}$ be a direction, i.e. $e \in \RN $, $|e|=1$,  and let us define the set 
 $$   \gO (e)= \{ x \in \gO : x \cdot e >0\}  $$

\begin{theorem} \label{fconvessa} Let $\gO$ be a ball or an annulus in $\RN $, $N \geq 2$, and let $U \in C^{3,\ga }(\ov {\gO}; \Rm)$ be a solution of \eqref{modprob} with Morse index $\gm (U) \leq N $. 
Moreover assume that:
\begin{itemize} 
\item [i) ]  The system is fully coupled along $U$ in $ \gO (e) $, for any $e\in S^{N-1}$.
 \item [ii) ] For any $i,j=1, \dots m $  \  $\frac {\de f_i} {\de u_j}(|x|,u_1, \dots , u_m)$ is nondecreasing in each variable  $u_k$,  $k=1, \dots , m $, for any $|x| \in \gO $.
\item [iii)]    if $m \geq 3 $ then,   for any $i\in  \{1, \dots ,m \}$,   $f_i(|x|,u_1,\dots , u_m)= \sum _{k \neq i} g_{ik}(|x|,u_i,u_k)$ where 
$g_{ik} \in C^{1, \ga }([0,+\infty )\times \R ^2)$.
  \end{itemize}
 Then $U$ is  foliated Schwarz symmetric and  if $U$ is not radial then it is strictly decreasing in the angular variable (see Definition \ref{foliatedSS}).
\end{theorem} 
 
  \begin{remark} \label{significatocondizioni}
  \begin{itemize}
 \item    Requiring that the system is fully coupled along $U$ in $\gO (e)$,  for any $e\in S^{N-1}$, is an assumption easily satisfied by most of the systems encountered in applications (see Section 4). It is essentially needed to ensure the validity of the strong maximum principle in the domains $\gO (e)$.
\item The monotonicity hypothesis ii) on the derivatives $\frac {\de f_i} {\de u_j}$ in Theorem \ref{fconvessa} implies that
 each $f_i$ is convex with respect to each variable $u_j $, $i,j=1,\dots ,m $.  \par
 \item The assumption iii), which states that each $f_i$ is the sum of functions which depend only on $u_i$ and one of the other variables $u_k$, is obviously not needed
 if $m=2$. \par
 Hypotheses  ii) and iii) together imply the following conditions, that will be exploited in the proof:
  \begin{itemize} 
\item [a) ]  $\frac {\de f_i} {\de u_i}(|x|,u_1, \dots , u_m)$ ( $= \sum _{k \neq i} \frac {\de g_{ik}}{\de u_i}(|x|,u_i,u_k)$ if $ m \geq 3$ ) is  nondecreasing in $u_i$ for any $i=1, \dots ,m$, i.e.  
if $U=(u_1, \dots , u_i, \dots ,u_m)$ and $\overline{U}= (u_1, \dots , \overline{u}_i, \dots ,u_m) $ with $u_i \leq \overline{u}_i$  then $\frac {\de f_i} {\de u_i}(|x|,U) \leq 
\frac {\de f_i} {\de u_i} (|x|,\overline{U}) $ for any $x \in \gO $.
\item [b) ] \  If $i\neq j $ \  $\frac {\de f_i} {\de u_j}(|x|,u_1, \dots , u_m) $ ( $= \frac {\de g_{ij}} {\de u_j}(|x|,u_i,  u_j)$ if $ m \geq 3$  ) depends only on $u_i$, $u_j$  and it is  nondecreasing in $u_i, u_j $, i.e.  if  $U=(u_1, \dots , u_m)$, $V=(v_1, \dots , v_m)$ and $u_i \leq v_i$, $ u_j \leq v_j$,  then 
$ \frac {\de f_i} {\de u_j}(|x|,U) \leq 
\frac {\de f_i} {\de u_j} (|x|, V) $ for any $i,j=1, \dots , m $ for any $x \in \gO $.
 \end{itemize}
 \end{itemize}
 \end{remark}
 Obviously the previous theorem holds in particular for stable solutions. However in this case, as for scalar equations, it is not difficult to see that the solution is radial, even without the assumption ii) and iii). 
 Therefore we have the following
 \begin{theorem}  \label{teorema2}
 If the system is fully coupled  along a stable solution $U$ in $\gO$, then $U$ is radial.
  \end{theorem}
 
 We shall see in the proof of Theorem \ref{fconvessa} that for nonradial Morse index one solutions the following conditions hold.
 \begin{corollary} \label{corollario1}
 Under the assumptions of Theorem \ref{fconvessa} if a solution $U$ has Morse index one and it is not radial then necessarily
\be \label{superfullycoupling1}
 \sum_{j=1}^m  \frac {\de f _i}{\de u _j}(r, U(r, \theta)) \frac {\de u_j } {\de \theta  }(r, \theta )=
 \sum_{j=1}^m  \frac {\de f _j}{\de u _i}(r, U(r, \theta))   \frac {\de u_j } {\de \theta  }(r, \theta )
\ee
 for any    $i=1,\dots ,m $, whith $(r, \theta)$ as in Definition \ref{foliatedSS}. \\ 
   In particular  if $m=2$ then \eqref{superfullycoupling1} implies that 
 \be \label{superfullycoupling2} 
 \frac {\de f _1}{\de u _2}(|x|, U(x))=  \frac {\de f _2}{\de u _1}(|x|, U(x)) \; , \quad \forall \, x \in \gO \;
 \ee
 \end{corollary}

 \begin{remark} The result of Corollary \ref{corollario1} is somewhat surprising because it asserts, under the hypotheses of Theorem \ref{fconvessa}, that for any nonradial Morse index one solution another coupling condition, namely \eqref{superfullycoupling1} (and in particular  \eqref{superfullycoupling2} if $m=2$), must hold along the solution. Some consequences of this, for a particular system, will be illustrated in Section 4. 
  \end{remark}
  
  Let us now explain the strategy of the proof of Theorem \ref{fconvessa}.\par
  As in the scalar case the key idea, introduced in \cite{P}, to get the symmetry of the solution, exploiting the information on its Morse index, is to use convexity assumptions on the nonlinearity to relate the linearized operator at the solution with the linear operator that arises when considering the difference between the solution and its reflection with respect to an hyperplane passing through the origin (see \cite{P},  \cite{PR}, \cite{PW}, \cite{GPW}). To do this, in the scalar case and for bounded domains $\gO$, the information on the quadratic form associated to the linearized operator at the solution, deduced by its Morse index, is exploited by means of the eigenvalues of the linearized operator. Indeed this operator is selfadjoint and so there is a variational characterization of the eigenvalues which links them to the quadratic form.
  This is partially used also in the unbounded domain case (see \cite{GPW}) even if the spectrum cannot be easily described, exploiting the eigenvalues  of the linearized operator in large balls.\par
  In the case of systems this no longer holds. Indeed if the Jacobian matrix $J_F (|x|,U)$ is not symmetric the linearized operator $L_U$ (see \eqref{linearizedoperator}) is not selfadjoint, even in  bounded domains. Then no variational characterization of eigenvalues can be exploited and hence the scalar case approach cannot be straigtforward followed.\par
  To bypass this difficulty we associate to the Jacobian matrix $J_F (|x|,U)$ its symmetric part $\frac 12 (J_F (|x|,U)+J_F^t (|x|,U) $, where $J_F^t $ is the transpose of the matrix $J_F$ (see \eqref{matricesimmetricaassociata} in the linear case), which then gives rise to the selfadjoint operator $\tilde{L}_U$ (see \eqref{symmetriclinearizedoperator} ) whose spectrum can be variationally characterized.\par
  The crucial, simple remark is that the quadratic form associated to the linearized operator $L_U$  is the same as the quadratic form associated to the selfadjoint operator $\tilde{L}_U$. \par
  Therefore the eigenvalues of $\tilde{L}_U$ can be exploited to study the symmetry of the solution $U$, using the information on the Morse index. 
  However, though essential in our proofs, the consideration of these eigenvalues is not sufficient to reach the conclusion of our theorems, but we also need to use the principal eigenvalue of the linearized operator, that can be defined for general linear operators (see \cite{BNV}, \cite{BS2},\cite{Si}). This can be understood by the fact that the positivity of this  eigenvalue is a necessary and sufficient condition for the (weak) maximum principle to hold (see Proposition \ref{principaleigenvalue}) and the maximum principle is another key-ingredient in our proofs.\par
  Obviously if the jacobian matrix $J_F (|x|,U)$ is symmetric then the linearized operator is selfadjoint and the classical spectral theory holds, in particular the principal eigenvalue coincides with the first eigenvalue and all proofs are simpler. \\
  This happens, for example, when the system is of \emph{gradient type}, i.e. when $F= \text{grad } g$ for some scalar function $g$ (see \cite{deF2}). In this case the linearized operator corresponds to the second derivative of a suitable associated functional.\par
  However this is not the case for many interesting systems, like e.g. the so called  \emph{hamiltonian systems} (see \cite{deF2} and the references therein). \par 
       Another important remark about the  assumptions of our theorems is that, as in previous symmetry results for systems (see \cite{BS1}, \cite{deF1}, \cite{KP}), coupling conditions on the system are necessary. Indeed it is easy to construct counterexamples to the symmetry if they do not hold. \par In particular, since the definition of foliated Schwarz symmetry for vector-valued functions requires that all components are foliated Schwarz symmetric with respect to the same vector, a \emph{strong} coupling condition is needed.\par
  Finally let us make a general comment on symmetry theorems of the type described in this paper. \\
  As for the scalar case, to apply our symmetry results, some information on the Morse index of the solution is needed.  So the question is: how to get it? \par
  If the system is of gradient type     then, often, variational methods, used to prove the existence of solutions, also carry information on the Morse index (\cite{AC}, \cite{deF2}, \cite{MMP}). A standard example is given by solutions obtained by the Mountain Pass theorem.\\
  If the system is not of this type it could happen that using variational methods, one ends up with the second derivative of a functional which is strongly indefinite. This is for example the case of the so called Hamiltonian systems. \par 
  However this does not mean that solutions do not have finite (linearized) Morse index, because the linearized operator is not the second derivative of such functionals.  
Then the linearized operators and its Morse index could be studied by some other methods, like continuation techniques. We give some examples in Section 4.\par
  Since the study of the linearized system is also important to understand the orbital stability of the solution and informations on the Morse index are useful to get qualitative properties of solutions,  we believe that further investigation in this direction should be done.\par
  The outline of the paper is the following. In Section 2 we state and prove some important results on the spectral theory and maximum principles for linear systems. Most of them are either well known or could be easily deduced by known theorems, but we think that is worth to recall and collect them together to make the paper self-contained and to avoid to give vague references to the reader.\\
  In Section 3 we show preliminary results on our semilinear system \eqref{modprob} and prove our symmetry results.\\
  In Section 4 we present a few examples and make further comments.

 

\section{  \textbf{Preliminaries on linear systems: spectral theory and maximum principles} }
\label{se:2}
Let $\gO$ be any smooth bounded domain in $\RN$, $N \geq 2$, and $D$  a $m \times m $ matrix with bounded entries:
\begin{equation}\label{ipotesiD} D= \left (  d_{ij} \right ) _{i,j=1}^m \; , \; d_{ij} \in L^{\infty} (\gO)
\end{equation}
Let us consider the linear elliptic  system
\begin{equation}  \label{linearsystem} 
\begin{cases} - \gD U + D(x) U = F \quad &\text{in } \gO \\
U= 0 \quad & \text{on } \de  \gO
\end{cases}
\end{equation}
i.e. 
$$
\begin{cases}
 - \gD u_1 +  d_{11} u_1+  \dots   +d_{1m} u_m =f_1  & \text{ in } \gO  \\
  \dots  \dots & \dots \\ 
 - \gD u_m +  d_{m1} u_1+  \dots   +d_{mm} u_m =f_m  & \text{ in } \gO \\
  u_1=  \dots  = u_m   =0  & \text {  on } \de \gO 
\end{cases}
$$
 where $F=(f_1, \dots , f_m) \in (L^2 (\gO))^m $,  $U=(U_1, \dots , U_m) $. \\
 This kind of linear system  appears in the linearization of the semilinear  elliptic system  \eqref{modprob}.
 \begin{definition}\label{sistemilineariaccoppiati} 
 The system \eqref{linearsystem} is said to be
 \begin{itemize}
 \item  \emph{cooperative} or \emph{weakly coupled} in $ \gO $ if 
 \begin{equation} \label{weaklycoupled} d_{ij}  \leq 0 \, \text{a.e. in }\gO , \quad \text{whenever } i \neq j
 \end{equation}
 \item \emph{fully coupled} in $ \gO $ if it is weakly coupled in $\gO $ and the following condition holds:
 \begin{equation}\label{fullycoupled}
 \begin{split}
  \forall \, I,J \subset \{ 1, \dots , m \}\, ,\, & I,J \neq \emptyset \, , \, I \cap J = \emptyset \, , \, I \cup J = \{ 1, \dots , m \} \, \\
  &  \exists i_0 \in I \, , \, j_0 \in J \, : \text{meas } (\{ x \in \gO  : d_{i_0j_0} <0 \}) >0
  \end{split}
 \end{equation}
 \end{itemize}
 \end{definition}
It is well known that either condition \eqref{weaklycoupled} or conditions \eqref{weaklycoupled} and \eqref{fullycoupled} together are needed in the proofs of maximum principles for systems (see \cite{deF1}, \cite{deFM}, \cite{Si} and the references therein).
In particular if both are fulfilled the strong maximum principle holds as it is shown in the next  theorem.\par
 \textbf{Notation remark}:  here and in the sequel inequalities involving vectors should be understood  to hold componentwise, e.g. if  $\Psi = (\psi _1 , \dots , \psi _m) $,   $\Psi $ nonnegative  means that $\psi _j \geq 0 $ for any index $j=1, \dots , m$. 
\begin{theorem}\label{SMP} (Strong Maximum Principle and Hopf's Lemma).  Suppose that \eqref{ipotesiD},   \eqref{weaklycoupled} and \eqref{fullycoupled} hold  and $U=(u_1, \dots , u_m) \in C^1 (\ov {\gO};\Rm)$  is a weak solution of the inequality 
$$- \gD U + D(x) U \geq 0 \text{ in } \gO
$$ 
  i.e. 
\begin{equation} 
 \int _{\gO}  \nabla U \cdot  \nabla \Psi   + D(x) (U ,\Psi )= 
 \int _{\gO}\left [   \sum _{i=1}^m  \nabla u_i \cdot  \nabla \psi _i  +\sum _{i,j=1}^m d_{ij}(x) u_i  \psi _j \right ]  \, dx \geq 0
\end{equation}
for any nonnegative $\Psi = (\psi _1 , \dots , \psi _m) \in C_c^1 (\gO  ; \Rm)$.  \\
If $U \geq  0 $ in $\gO $, then either $U \equiv 0 $ in $\gO $ or $U>0 $ in $\gO $. In the latter case if $P \in \partial \gO $ and $U(P)=0$ then $\frac {\de U }{\de \gn }(P) < 0$, where $\gn $ is the unit exterior normal vector at $P$.
\end{theorem}
\begin{proof} It suffices to apply the scalar strong maximum principle using  \eqref{weaklycoupled} to obtain that  each component $u_k$ is either identically equal to zero or positive, and then to use  \eqref{fullycoupled} to prove that the same alternative holds for all the other components. To be more precise, for any equation we have 
$- \gD u_j + d_{jj}u_j \geq  \sum _{i\neq j} - d_{ij} u_i \geq 0 $, since $d_{ij} \leq 0 $ if $i\neq j$ and $u_i \geq 0$. This implies that, for any $j=1, \dots ,m$, either $u_j \equiv 0 $ or $u_j >0$ and in the latter case by Hopf's Lemma we have the sign of the normal derivative on  a point of the boundary where the function $u_j$ vanishes.\\ 
Suppose now that $U$ does not vanish identically in $\gO $ and let $J \subset \{1, \dots ,m \}$ the set of indexes $j$ such that $u_j>0 $ in $\gO$. Suppose by contradiction that  $J$ is a proper  subset of $ \{1, \dots ,m \}$, and let $ I  = \{1, \dots ,m \} \setminus  J $ the set of indexes $i$ such that $u_i \equiv 0$. If $i_0$, $j_0$ are as in \eqref{fullycoupled} then 
$- \gD u_{i_0} + d_{i_0 i_0 }u_{i_0} \geq - \sum _{j\neq i_0 } - d_{i_0 j} u_j \geq - d_{ i_{0}, j_{0}} u_{j_0}  \not \equiv  0 $, so that by the scalar strong maximum principle we get $u_{i_0}>0$, which is a contradiction.
\end{proof}
Let us recall that weak maximum principles, Harnack inequalities and other estimates have been studied in many papers with general conditions and also for elliptic operators not in divergence form (see  \cite{Si} and the references therein).
\\
We are interested in the quadratic form associated with system \eqref{linearsystem}, namely
\begin{equation} \label{formaquadraticalineari}
\begin{split}
Q (\Psi ; \gO ) &= \int _{\gO} \left [ |\nabla \Psi |^2 + D(x) (\Psi ,\Psi )\right ] dx= \\
& \int _{\gO}\left [   \sum _{i=1}^m |\nabla \Psi _i |^2 +\sum _{i,j=1}^m  d_{ij}(x)  \Psi _i \Psi _j \right ]  \, dx 
\end{split}
\end{equation}
for  $\Psi \in C_c^1 (\gO;\Rm)$ ( or $\Psi \in  H_0^1 (\gO;\Rm)$ ).\\
It is easy to see that this quadratic form coincides with the quadratic form  associated to the symmetric system
\begin{equation}  \label{symmetricsystem}  
\begin{cases} - \gD U + C(x) U = F \quad &\text{in } \gO \\
U= 0 \quad & \text{on } \de  \gO
\end{cases}
\end{equation}
i.e. 
$$
\begin{cases}
 - \gD u_1 +  c_{11} u_1+  \dots   +c_{1m} u_m &=f_1 \\
  \dots  \dots & \dots \\ 
 - \gD u_m +  c_{m1} u_1+  \dots   +c_{mm} u_m &=f_m
\end{cases}
$$
 where 
 \begin{equation}\label{matricesimmetricaassociata} C=  \frac 12 (D + D^{t}) \quad  \text { i.e. } \quad C=(c_{ij}), \quad     c_{ij} =  \frac 12 \, (d_{ij} + d_{ji})
 \end{equation}
So to study the sign of the quadratic form $Q$ we can also use the properties of the symmetric system.\\
Therefore we review briefly the spectral theory for this kind of simmetric systems, and use it to prove some results that we need for the possible nonsymmetric system  \eqref{linearsystem}.
\begin{remark} If system  \eqref{linearsystem} is cooperative, respectively fully coupled, so is  the associate symmetric system  \eqref{symmetricsystem}. 
\end{remark}
 \subsection{Spectral theory for symmetric systems}
 Let $\gO$ be a bounded domain in $\RN$, $N \geq 2$, and consider for $m \geq 1$ the Hilbert spaces 
 $\ldue = \textbf{L}^2 (\gO ) = \left ( L^2(\gO) \right )^m$, $\huno = \textbf{H}_0^1 (\gO)= \left ( H_0^1(\gO) \right )^m$, where if $f=(f_1, \dots , f_m)$, $g=(g_1, \dots , g_m)$ the scalar products are defined by 
 \be 
 \begin{split}
 &(f,g)_{\ldue}= \sum_{i=1}^m (f_i ,g_i )_{L^2(\gO)}=  \sum_{i=1}^m \int _{\gO} f_i \, g_i  \, dx \\ 
 &(f,g)_{\huno}= \sum_{i=1}^m (f_i ,g_i )_{H_0^1(\gO)}=  \sum_{i=1}^m \int _{\gO} \nabla f_i \,\cdot \, \nabla g_i  \, dx
\end{split}
 \ee
Let $C=C(x)=(c_{ij}(x))_{i,j=1}^m $ a symmetric matrix whose elements are bounded functions:
\be \label{ipotesiC} c_{ij} \in L^{\infty} \quad , \quad c_{ij}=c_{ji} \quad \text{ a.e. in }  \gO
\ee  
 and consider the bilinear forms
 \be \label{bilinearform} B(U,\Phi)= \int _{\gO} \left [ \nabla U \cdot \nabla \Phi + C(U, \Phi) \right ] = \int _{\gO} \left [\sum_{i=1}^m  \nabla u_i \cdot \nabla \phi _i +\sum _{i,j=1}^m c_{ij}u_i \phi _j \right ] 
 \ee
 and, for $\gL >0 $  
 \be 
 \begin{split}
 &B^{\gL}(U,\Phi)= \int _{\gO} \left [ \nabla U \cdot \nabla \Phi + (C+\gL I )(U, \Phi) \right ] \\
 &= \int _{\gO} \left [\sum_{i=1}^m  \left (\nabla u_i \cdot \nabla \phi _i  + \gL u_i  \phi _i\right ) +\sum _{i,j=1}^m c_{ij}u_i \phi _j \right ] 
 \end{split}
 \ee
 Since $c_{ij}\in L^{\infty}$ and $c_{ij}=c_{ji}$,  $B$ and $B^{\gL}$ are continuous symmetric bilinear forms, and, since $|\int _{\gO}c_{ij}u_i \phi _j | \leq c \int _{\gO} (u_i^2 + \phi _j ^2)$, there exists $\gL \geq 0 $ such that $ B^{\gL}$ is coercive in  $ \huno $, i.e it is an equivalent scalar product in $\huno $.\\
 By the Riesz representation theorem, identifying $F\in \ldue$ with the linear functional \ \ $U \in \huno \mapsto (U,F)_{\ldue}$ , \  \ 
 for any $F \in \ldue $ there exists a unique $U=: T\, F \in \huno $ such that $\Vert U \Vert _{\huno} \leq c \Vert f \Vert _{\ldue} $ and 
 $B^{\gL}(U,\Phi) = (F,U)_{\ldue} $ for any $\Phi \in \huno$, i.e. $U$ is the unique weak solution of the system
  \begin{equation} 
\begin{cases} - \gD U + \left (C(x) + \gL I \right ) U = F \quad &\text{in } \gO \\
U= 0 \quad & \text{on } \de  \gO
\end{cases}
\end{equation}
i.e.
 $$ 
 \begin{cases}   - \gD u_1 +  (c_{11} + \gL ) u_1+  \dots   +c_{1m} u_m  = f_1  \quad  & \text{ in } \gO  \\
   \dots  &  \dots \\ 
 - \gD u_m +  c_{m1} u_1+  \dots   +(c_{mm} + \gL ) u_m  = f_m  \quad  & \text{ in } \gO \\
   u_1=  \dots  = u_m   =0  & \text {  on } \de \gO 
 \end{cases}
 $$

 Moreover $T: F \mapsto U $, maps  $\ldue $ into $ \ldue $ and  is \emph{compact} because of the compact embedding of $\huno $ in $\ldue $,
 it is a \emph{positive} operator, since $ (TF,F)_{\ldue} = (U,F)_{\ldue}= B^{\gL}(U,U) >0  $ if $F \neq 0$ which implies $U \neq 0$ (recall that $ B^{\gL} $ is an equivalent scalar product in $\huno $ ), and,  since $C$ is symmetric, it is also \emph{selfadjoint}. Indeed 
 if $TF=U$, $TG=V$, i.e $B^{\gL}(U,\Phi) = (F,U)_{\ldue} $,   $B^{\gL}(V,\Phi) = (G,V)_{\ldue} $ for any $\Phi \in \huno$, then
 $(TF,G)_{\ldue} =(U,G)_{\ldue} = (G,U)_{\ldue} =B^{\gL}(V,U)  =B^{\gL}(U,V)=(F,V)_{\ldue} = (F,TG)_{\ldue}   $.\\
 Thus, by the spectral theory of positive compact selfadjoint operators in Hilbert spaces there exist a nonincreasing sequence $\{ \gm _j^{\gL} \} $ of eigenvalues with 
 $\lim _{j \to \infty}\gm _j^{\gL}=0 $ and a corresponding sequence $\{ W^j \} \subset \huno $ of eigenvectors such that $G(W^j)= \gm _j^{\gL} W^j $.
 Putting $  \gl _j^{\gL} = \frac 1 { \gm _j^{\gL}  }$ then $W_j$ solves the systems
 $ - \gD W^j + (C + \gL I) W^j =   \gl _j^{\gL} W^j, \quad W^j =0 \text{ on } \de \gO
 $. 
 Translating, and denoting by $\gl _j$ the differences $\gl _j=  \gl _j^{\gL} -\gL  $, we conclude that there exist a sequence $\{\gl _j \}$ of eigenvalues, with
 $- \infty < \gl _1 \leq \gl _2 \leq \dots $, $\lim _{j \to + \infty} \gl _j = + \infty $, and a corresponding sequence of eigenfunctions $\{ W^j \}$ that weakly solve the systems
 \be \label{sistemaautovalori}
 \begin{cases} - \gD W^j + C  W^j =   \gl _j W^j &  \text{ in }  \gO \\
  W^j =0 & \text{ on } \de \gO
 \end{cases}
 \ee
i.e. if $W^j=(w_1, \dots , w_m)$
$$
\begin{cases}
 - \gD w_1 +  c_{11} w_1+  \dots   +c_{1m} w_m &= \gl _j  w_1 \\
  \dots  \dots & \dots \\ 
 - \gD w_m +  c_{m1} w_1+  \dots   +c_{mm} w_m &=\gl _j  w_m
\end{cases}
$$ 
Moreover by (scalar) elliptic regularity theory applied iteratively to each equation, the eigenfunctions $W^j $ belong at least to $  C^{1}(\gO;\Rm)$.\\
We now  collect in the next proposition  the variational formulation and some properties of eigenvalues and eigenfunctions.\\
 In what follows if $\gO '$ is a subdomain of $\gO $ we denote by $\gl _k (\gO ')$ the eigenvalues of the same system with $\gO $ substituted by $\gO '$.
 
\begin{proposition}\label{varformautov} Suppose that $C=(c_{ij})_{i,j=1}^m $ satisfies \eqref{ipotesiC}, and let $\{\gl _j \}$, $\{ W^j \}$ be the sequences of eigenvalues and eigenfunctions that satisfy \eqref{sistemaautovalori}.\\
Define the Rayleigh quotient
\be R(V)= \frac {B(V,V)}{(V,V)_{\ldue}} \quad \text{ for } V \in \huno \quad V \neq 0
\ee
with $B(.,.)$ as in \eqref{bilinearform}
Then the following properties hold, where  $\textbf{V}_k$ denotes a $k$-dimensional subspace of $\huno$ and the orthogonality conditions  $V \bot W_k $  or  $V \bot \textbf{V}_k $  stand for the orthogonality in $\ldue$.  
\begin{itemize}
\item[i)] $\gl _1 = \text{ min } _{V\in \huno \,, \, V \neq 0} R(V) =  \text{ min } _{V\in \huno \, , \,  (V,V)_{\ldue}=1 } B(V,V) $
\item[ii)]   $\gl _m = \text{ min } _{V\in \huno \,, \, V \neq 0 \, , \, V \bot W_1,\dots , V \bot W_{m-1}} R(V)$\\ $ =  \text{ min } _{V\in \huno \, , \,  (V,V)_{\ldue}=1 \, , \, V \bot W_1,\dots , V \bot W_{m-1} } B(V,V) $  if $m \geq 2$ 
\item[iii)] $\gl _m = \text{ min } _{\textbf{V}_m }   \text{ max } _{ V \in \textbf{V}_m \, , \, V \neq 0 } R(V) $
\item[iv)] $\gl _m = \text{ max } _{\textbf{V}_{m-1} }   \text{ min } _{ V \bot  \in \textbf{V}_{m-1 }\, , \, V \neq 0 } R(V) $
\item[v)] If $W \in \huno$, $W \neq 0$, and $R(W)= \gl _1$, then $W$ is an eigenfunction corresponding to $\gl _1$.
\item[vi)] $ \lim _{\text{ meas }(\gO ') \to 0} \gl _1 (\gO ') = + \infty $
\item[vii)] If the system is cooperative in $\gO $ and $W $ is a first eigenfunction, then $W^{+}$ and $W^{-}$ are eigenfunctions, if they do not vanish.
\item[viii)] If the system is fully coupled in $\gO $,  then the first eigenfunction does not change sign in $\gO $ and  the first eigenvalue is simple, i.e. up to scalar multiplication there is only one eigenfunction corresponding to the first eigenvalue. 
\item[ix)] Assume that the system is fully coupled in $\gO $,  $C'  = \left (  c'_{ij} \right )_{i,j=1}^m $ is another matrix that satisfies \eqref{ipotesiC}, and let $\{{ \gl '} _k  \}$, 
 be the sequence of eigenvalues  of  the corresponding system.
If $c_{ij} \geq  {c'} _{ij}  $ for any $i,j=1, \dots ,m$ then  $ \gl _1 \geq { \gl '} _1$.
\end{itemize}
\end{proposition}

\begin{proof} As before let us consider for   $\gL \geq 0 $ the bilinear form $ B_{\gL }(V,V)= B(V,V) + \gL (V,V)_{\ldue}$, which is an equivalent scalar product in $\huno$, and  define $ R_{\gL}(V)= \frac {B_{\gL }(V,V)}{(V,V)_{\ldue}} \quad \text{ for } V \in \huno \quad V \neq 0 $. 
  Since  $  R_{\gL}(V) = R(V) + \gL $, once the properties are proved for  $  B_{\gL}(V) $ (which is an equivalent scalar product in $\huno $)  and its eigenvalues $\gL _j^{\gL}$, we recover the results stated by translation. Therefore for simplicity of notations we assume from the beginning that $\gL =0$ i.e. that $B(.,.)$ is an equivalent scalar product in $\huno $.\\
The proofs of i) ,\dots , vi) do not depend on cooperativeness and they are the same as the standard proofs in the scalar case (see \cite{E}, \cite{K}), so we only sketch them.\\
i) \ The sequence $\{ W_j \}$ is an orthonormal basis of $\ldue  $, and  since $B(W_k,W_j)= \gL _k (W_k,W_j)_{\ldue}$ (in particular $=0$ if $k\neq j $), the sequence  
$\{ (\gL _j ) ^{-\frac 12}\, W_j \}$  is an orthonormal basis of $\huno   $. It follows that if $U = \sum _{k=1}^{\infty } d_k W_j $ is the Fourier expansion of a function $U$ in 
$\ldue $, the series converges to $U$ in $\huno $ as well. 
If now $(U,U)_{\ldue} = \sum _{k=1}^{\infty } d_k^2 =1$, then 
$B(U,U)= \sum _k \gl _k d_k ^2 \geq \gl _1 \sum _k d_k ^2= \gl _1 $ and i) follows.\\
ii) \ If $ V \bot W_1,\dots , W_{m-1} $ and $(V,V)_{\ldue}=1)$, then $V = \sum _{k=m}^{\infty } d_k W_j $ and as before $B(V,V) \geq \gl _m$ and since $B(W_m, W_m)= \gl _m$ ii) follows.\\
iii) \ If $\text{ dim }(\textbf{V}_m) =m $ and $\{ V_1, \dots , V_m \}$ is a basis of $\textbf{V}_m  $, there exists a linear combination $0 \neq V=\sum _{i=1}^m \ga _i V_i$ which is orthogonal to $W_1$, \dots $W_{m-1}$ ($m$ coefficients and $m-1$ unknown), so that by ii) we obtain that
$ \text{ max } _{ V \in \textbf{V}_m \, , \, V \neq 0 } R(V) \geq \gl _m$. On the other hand if $\textbf{V}_m= \text{ span }(W_1, W_m)$ then 
$ \text{ max } _{ V \in \textbf{V}_m \, , \, V \neq 0 } R(V) \leq  \gl _m$, so that iii) follows.\\
iv) \  The proof  is similar. If $\{ V_1, \dots , V_{m-1}\} $ is a basis of an $m-1$-dimensional subspace $\textbf{V}_{m-1}$, there exists a linear combination
$0 \neq W=\sum _{i=1}^m \ga _i W_i$ of the first $m$ eigenfunctions which is orthogonal to $\textbf{V}_{m-1}$, and $R(W,W) \leq \gl _m $.
So  $\text{ min } _{ V \bot  \in \textbf{V}_{m-1 }\, , \, V \neq 0 } R(V) \leq \gl _m $, but taking $\textbf{V}_m= \text{ span }(W_1, W_{m-1})$ then  
$\text{ min } _{ V \bot  \in \textbf{V}_{m-1 }\, , \, V \neq 0 } R(V) \geq \gl _m $, so that iv) follows.\\
v) \ By normalizing we can suppose that $(W,W)_{\ldue}=1$. Let $V \in \huno $, $t>0$. Then by i) $R(W+tV)= \frac { B(W+tV,W+tv)}{(W+tV)_{\ldue}}\geq \gl _1$, i.e.
$B(W,W)+ t^2 B(V,V) + 2t B(W,V) \geq \gl _1 \left [ (W,W)_{\ldue}  + t^2 (V,V)_{\ldue}  + 2t (W,V)_{\ldue}  \right ] = \gl _1 + \gl _1 t^2 (V,V) + 2t \gl _1 (W,V)$. Since $B(W,W) = \gl _1$, dividing by $t$ and letting $t \to 0$ we obtain that $B(W,V) \geq \gl _1 (W,V)_{\ldue} $ and changing $V$ with $-V$ we deduce that 
$B(W,V) = \gl _1 (W,V)_{\ldue} $ for any $V \in \huno $, i.e. $W$ is a first eigenfunction. \\
vi) \  If $V \in \textbf{H}_0^1 (\gO ')$ there exists $C \geq 0$ such that \par 
$B_{\gO '} (V,V) \geq \int _{\gO '} |\nabla V |^2 \, dx - C  \int _{\gO '} |V|^2 \, dx $ , \par while by  Poincar\'e' s inequality  \par 
$\int _{\gO '}|V|^2 \, dx \leq C' |\gO '| ^{\frac 2N} \int _{\gO '} |\nabla V |^2 \, dx $,  so that \par 
 $R_{\gO '}(V)= \frac {B_{\gO '} (V,V)}{\int _{\gO '}|V|^2} \geq \frac {  \int _{\gO '} |\nabla V |^2 \, dx } {  \int _{\gO '} |V|^2 \, dx }- C \geq 
\frac 1{C' |\gO '|^ {\frac 2N}} -C \to + \infty $ if $|\gO ' | \to 0 $.\\
vii) \ Multiply  \eqref{sistemaautovalori} by $W^{+}=(w_1^{+}, \dots , w_m^{+}) $ and integrate. If in the $i$-th  equation, multiplied by $w_i^{+}$ we write $w_j= w_j^{+}-w_j^{-}$ for $j\neq i$ and recall that by cooperativeness  $ - c_{1j}w_j^{-}w_1^{+} \geq 0 $, we deduce that 
$B(W^{+},W^{+})\leq \gl _1 (W^{+}, W^{+})$, so that by v) $W^{+}$ is a first eigenfunction. The same applies to  $W^{-}$ \\
viii) \ The conclusion follows from the strong maximum principle, which is valid under the fully coupling hypothesis. In fact  if $W^{+}$ does not vanish, it is a first eigenfunction by vii), and by the strong maximum principle (Theorem \ref{SMP}) is strictly positive in $\gO$, i.e. $W >0 $ in $\gO $ if it is positive somewhere.\\
If $W^1$, $W^2$ are two eigenfunctions corresponding to $\gl _1$, they do not change sign in $\gO $, so that they can not be orthogonal in $\ldue $. This implies that the first eigenvalue is simple.\\
ix) Let $W^1=(w_1, \dots , w_m )$ the first eigenfunction for the system \eqref{sistemaautovalori}. Since the system is fully coupled, $W^1$ does not change sign, and by normalizing it, we can assume that  $(W^1, W^1)_{\ldue}=1$. \\
Denoting by $B '$ the bilinear form corresponding to the matrix $ C' $, since  $w_i  w_j \geq 0 $ and $c_{ij} \geq  {c'}_{ij}  $ we get that
\begin{multline}
 \gl _1 = B(W^1,W^1)=  \int _{\gO} \left [\sum_{i=1}^m  |\nabla w_i|^2   +\sum _{i,j=1}^m c_{ij} w_i  w_j \right ] \, dx  \geq \\
\int _{\gO} \left [\sum_{i=1}^m  |\nabla w_i|^2   +\sum _{i,j=1}^m  {c'}_{ij} w_i  w_j \right ] \, dx = B' (W^1,W^1) \geq {\gl _1}'
\end{multline}

\end{proof}
 \subsection{Weak Maximum Principles for Cooperative systems}
 Let us turn back to the (possibly) nonsymmetric  cooperative system  \eqref{linearsystem}:\\
 $$ \begin{cases} - \gD U + D(x) U = F \quad &\text{in } \gO \\
U= 0 \quad & \text{on } \de  \gO
\end{cases}
  $$ 
where the  matrix $D=(d_{ij})_{i,j=1}^m$  satisfies 
\be \label{ipotesiDcooperativa}
d_{ij} \in L^{\infty} (\gO) \quad, \quad   d_{ij} \leq 0 \, , \quad \text{whenever } i \neq j 
\ee
In the sequel we shall indicate by $ \gl _j^{\text{(s)}}=  \gl _j^{\text{(s)}}(-\gD +D; \gO )    $ the  eigenvalues of the associated symmetric system: \\
 \eqref{symmetricsystem} \ \ 
$ \begin{cases} - \gD U + C(x) U = F \quad &\text{in } \gO  \\
U= 0 \quad & \text{on } \de  \gO 
\end{cases}
 $ \\
 where \ 
 $C=  \frac 12 (D + D^{t}) \;  \text { i.e. } \;    c_{ij} =  \frac 12 \, (d_{ij} + d_{ji}) $. \\
 Analogously the corresponding eigenfunctions will be indicated by $W_j^{\text(s)}$.\\
 We also  denote the bilinear form associated with the symmetric system \eqref{symmetricsystem} by
   $$ B^{\text{s}}(U,\Phi)= \int _{\gO} \left [ \nabla U \cdot \nabla \Phi + C(U, \Phi) \right ] = \int _{\gO} \left [\sum_{i=1}^m  \nabla u_i \cdot \nabla \phi _i +\sum _{i,j=1}^m c_{ij}u_i \phi _j \right ]  
 $$
  As already remarked, the quadratic form \eqref{formaquadraticalineari} associated to the system \eqref{linearsystem} coincides with that associated to  system \eqref{symmetricsystem}, i.e. 
$$ Q (\Psi ; \gO  ) = \int _{\gO} |\nabla \Psi |^2 + D(x) (\Psi ,\Psi ) = B^{\text{s}}(\Psi,\Psi)
$$ 
  if   $\Psi \in  H_0^1 (\gO;\Rm)$. \par 
\medskip
\begin{definition} We say that the maximum principle holds for the operator $- \gD + D$ in an open set  $\gO ' \subseteq \gO $ if  any $U \in \textbf{H}^1 (\gO ')$ such that 
\begin{itemize}
\item $U \leq  0 $ on $ \de \gO ' $ (i.e. $U^{+} \in \textbf{H}_0^1 (\gO ')$ ) 
\item  $- \gD U + D(x)U \leq 0 $ in $\gO ' $ (i.e. 
$\int \nabla U \cdot  \nabla \Phi + D(x) (U, \Phi ) \leq 0 $
for any nonnnegative  $\Phi \in \textbf{H}_0^1 (\gO ') $ )
\end{itemize}
 satisfies $U \leq 0 $ a.e. in $\gO $.
\end{definition}
Let us denote by $  \gl _j^{\text{(s)}}  (\gO ') >0  $ the sequence of the eigenvalues of the symmetric system in an open set $\gO ' \subseteq \gO$.
\begin{theorem}\label{WMP} [Sufficient conditions for weak maximum principle] Under the hypothesis \eqref{ipotesiDcooperativa}, if
 $  \gl _1^{\text{(s)}}  (\gO ') >0  $ then the maximum principle holds for $- \gD + D$ in $\gO ' \subseteq \gO $. \\
\end{theorem}
\begin{proof} By the variational characterization of eigenvalues \\
$\gl _1 ^{\text{(s)}} = \text{ min } _{V\in \textbf{H}^1 (\gO ') \,, \, V \neq 0} R(V) > 0 $, so that $Q(V)=B^{\text{s}}(V,V) >0 $ for any $V \neq 0 $ in $ \textbf{H}_0^1 (\gO ') $.
Suppose that $U \leq  0 $ on $ \de \gO ' $, $- \gD U + D(x)U \leq 0 $ in $\gO ' $. Then, testing the equation with $U^{+}=(u_1^{+},\dots , u_m^{+})$, writing in the $i$-th equation 
$u_j= u_j^{+} - u_j^{-}$ for $i \neq j $, and recalling that $-c_{ij}u_i^{+} u_j^{-} \geq 0 $ if $i\neq j$, we obtain that 
$B^{\text{s}}(U^{+},U^{+}) \leq 0 $, which implies $U^{+} \equiv 0 $ in $ \gO ' $.
\end{proof}
 
 Almost immediate consequences of the previous theorem are the  following ''Classical'' and ''Small measure'' forms of the weak maximum principle (see   \cite{BS2}, \cite{deFM},  \cite{PrW}, \cite {Si}).

\begin{corollary}  
\begin{itemize}
\item[i) ] If \eqref{weaklycoupled} holds and  $D$ is a.e. nonnegative definite in $\gO ' $ then the maximum principle holds for
$- \gD + D$ in $ \gO ' $. 
\item[ii) ]
There exists $\gd >0 $, depending on $D$, such that for any subdomain $\gO ' \subseteq \gO $ the maximum principle holds for $- \gD + D$ in $\gO ' \subseteq \gO $ provided $|\gO ' | \leq \gd$. 
\end{itemize}
\end{corollary}
\begin{proof} i) \  If the matrix $D$ is  nonnegative definite then $B^{\text{s}}(\Psi , \Psi ) \geq \int _{\gO } |\nabla \Psi |^2 $, which implies that the first symmetric eigenvalue is positive, so that by Theorem \ref{WMP} we get i).\\   
ii) It is a consequence of the Poincar\' e 's inequality through vi) of  Proposition \ref{varformautov}. 
\end{proof}
 
Obviously the converse of Theorem \ref{WMP} holds if $D=C$ is symmetric: if the maximum principle holds for $- \gD + C$ in $\gO '$ then $\gl _1^{\text{(s)}} (\gO ' )>0$.
In fact if $\gl _1^{\text{(s)}} (\gO ' )\leq 0 $ since the system is cooperative (and symmetric) there exists a nontrivial nonnegative first eigenfunction $\Phi _1 \geq 0 $, 
$\Phi \not \equiv 0$, and  the maximum principle does not hold, since $- \gD \Phi _1 + C\, \Phi _1 = \gl _1 \Phi _1 \leq 0 $ in $\gO '$, $\Phi _1 =0 $ on $\de \gO '$, 
while  $\Phi _1 \geq 0 $ and
$\Phi _1 \neq 0$.\\
However this is not true for general nonsymmetric systems. Roughly speaking the reason is that there is an equivalence between the validity of the maximum principle for the operator $- \Delta + D$ and the positivity of its  principal eigenvalue $\tilde {\gl _1}$, whose definition  is given below, and  the inequality 
$\tilde {\gl _1} (\gO ') \geq    \gl _1^{\text{(s)}}  (\gO ' )  
$,  which can be  strict, holds. \par  
\medskip

More precisely we recall  that the \emph{principal eigenvalue}  of the operator \\
 $- \gD + D$ in an open set $\gO ' \subseteq \gO $ is defined as  
\be
\begin{split}
 \tilde {\gl _1} (\gO  ') &= 
   \sup \{ \gl \in \R : \exists \, \Psi \in W^{2,N}_{loc} (\gO ' ; \Rm) \; \text { s.t. } \\ 
&      \Psi  >0   - \gD \Psi + D(x) \Psi - \gl \Psi \geq 0 \text{ in } \gO '\}
\end{split}
\ee
(see \cite{BS2} and the references therein, and also \cite{BNV} for the case of scalar equations).\\
We then have:
\begin{proposition}\label{principaleigenvalue} Suppose that the system \eqref{linearsystem} is fully coupled in an open set  $\gO ' \subseteq \gO $. Then:
\begin{itemize}
\item[i) ] there exists a positive eigenfunction $\Psi _1 \in  W^{2,N}_{loc} (\gO ' ; \Rm)$ which satisfies  
\be \label{autofprincipale}   - \gD \Psi _1 + D(x) \Psi _1 = \tilde {\gl _1 } (\gO ') \Psi _1 \text { in } \gO\, , \quad \Psi _1 >0    \text { in }  \gO ' \, , \quad  \Psi _1 =0    \text { on } \de \gO ' 
\ee
Moreover the principal eigenvalue is simple, i.e. any function that satisfy \eqref{autofprincipale} must be a multiple of $\Psi _1$.
\item[ii) ]  the maximum principle holds for the operator $- \gD + D$ in $\gO '$ if and only if $\tilde {\gl _1} (\gO ' ) >0$
\item[iii) ] if there exists a positive function $\Psi \in  W^{2,N}_{loc} (\gO ' ; \Rm) $ such that   $\Psi  >0 , \,  - \gD \Psi + D(x) \Psi  \geq 0 $ in $\gO '$, then either 
$\tilde {\gl _1} (\gO  ') >0$ or $\tilde {\gl _1} (\gO ') =0 $ and $\Psi = c\, \Psi _1$ for some constant $c$.
\item[iv) ]  $\tilde {\gl _1} (\gO ' ) \geq  \gl _1^{\text{(s)}}  (\gO ' ) $, with equality if and only if $\Psi _1 $ is also the first eigenfunction of the symmetric operator $- \gD + C$ in $\gO '$ ,  $ C=\frac 12 (D + D^{\text{t}})$.
If  this is the case the equality $C(x) \Psi _1 =D(x) \Psi _1$ holds and,  if $m=2$, this implies that $d_{12}= d_{21}$.
\end{itemize}
\end{proposition}
\begin{proof} We refer to \cite{BS2}  for the proofs of i) --  iii). For what concerns iv) we observe that evaluating the quadratic form $Q(\Psi ; \gO ' )$ on $\Psi _1$ and recalling that it coincides with the quadratic form associated with the symmetric operator $- \gD + C$ ,  $ C=\frac 12 (D + D^{\text{t}})$ we obtain that  \\
$\tilde {\gl _1} (\gO  ') = \frac {B^{\text{s}}(\Psi _1 , \Psi _1 )}{(\Psi _1 , \Psi _1)_{\ldue}} \geq  \gl _1^{\text{(s)}}  (\gO ')  $ with equality if and only if $\Psi _1 $ is the first symmetric eigenfunction, by Proposition \ref{varformautov} v).
If this is the case, since $\Psi _1$ satisfies the system \eqref{autofprincipale} and the system \eqref{linearsystem}, the equality $D(x) \Psi _1 = C(x)\Psi _1$ follows. Since $\Psi _1 $ is positive, if $m=2$ the equality $d_{12}= d_{21} $ follows.
 \end{proof}

\section{  \textbf{Results on semilinear systems} }
 Let $\gO $ be a bounded domain in $\RN $, $F=(f_1, \dots , f_m ): \ov {\gO}\times \Rm \to \Rm $  a $C^{1, \alpha }$ function and  consider the semilinear elliptic system for the unknown vector valued function $U=(u_1, \dots ,u_m)$  in $\gO$ :
\begin{equation}  \label{semilinear system}
\begin{cases} - \gD U =F(x,U) \quad &\text{in } \gO \\
U= 0 \quad & \text{on } \de  \gO
\end{cases}
\end{equation}
i.e.
$$
\begin{cases}  
 - \gD u_1  = f_1(x,u_1,\dots , u_m) & \text{ in } \gO  \\
  \dots  \dots & \dots \\ 
 - \gD u_m = f_m(x,u_1,\dots , u_m)   & \text{ in } \gO \\
 u_1=  \dots = u_m  =0 & \text { on } \de \gO 
\end{cases}
$$
The system \eqref{modprob} is a particular case of this system, whith radial dependence on $x$.
\begin{definition} \label{semilinearicoupled}
\begin{itemize}
\item We say that the system \eqref{semilinear system} is cooperative or weakly coupled in an open set $\gO ' \subseteq \gO $ if 
$$  \frac {\de f _i} {\de u _j} (x, u_1, \dots , u_m) \geq 0 \quad \text{ for every } (x,u_1,\dots ,u_m ) \in  \gO ' \times \Rm
$$
and every  $i,j=1,\dots ,m$ with $i \neq j$.
 \item  We say that the system \eqref{semilinear system} is fully  coupled along a solution $U$ in  $\gO ' \subseteq \gO $  if it is cooperative in $\gO '$ and 
    in addition $\forall I,J \subset \{1, \dots ,m \}$ such that $I \neq \emptyset $, $J \neq \emptyset $, $I \cap J = \emptyset $, $I \cup J = \{1, \dots ,m \} $ there exist $i_0 \in I$, $j_0 \in J $ such that 
$$\text{meas }(\{ x \in \gO ' :   \frac {\de f _{i_0}} {\de u _{j_0}} (x, U(x)) > 0 \}) >0$$  
 \end{itemize}
\end{definition}

Let us recall the following definition, which makes sense also  in possibly unbounded domains.
\begin{definition}
\begin{itemize}
\item [i)]
Let $U$ be a $C^2 (\gO;\Rm)$ solution of \eqref{semilinear system}. We say that $U$ is linearized stable (or that has zero Morse index) if the quadratic form 
\begin{equation} \label{formaquadraticalinearizzato} 
\begin{split}
Q_U (\Psi ; \gO ) &= \int _{\gO} \left [ |\nabla \Psi |^2 - J_F (x, U)(\Psi ,\Psi )\right ] dx= \\
& \int _{\gO}\left [   \sum _{i=1}^m |\nabla \psi _i |^2 -\sum _{i,j=1}^m \frac {\de f_i}{\de u_j}(x, U(x)) \psi _i \psi _j \right ]  \, dx \geq 0
\end{split}
\end{equation}
for any $\Psi =(\psi _1, \dots , \psi _m) \in C_c^1 (\gO;\Rm)$ where $J_F (x,U)$ is the jacobian matrix of $F$ computed at $U$.
\item [ii)] $U$ has (linearized) Morse index  equal to the integer $\mu=\mu (U)\geq 1$ if $\mu $ is the maximal dimension of a subspace of  $ C_c^1 (\gO;\Rm)$ where the quadratic form is negative definite.
\item [iii)] $U$ has infinite (linearized) Morse index  if for any integer $k$ there is a $k$-dimensional subspace of  $ C_c^1 (\gO;\Rm)$ where the quadratic form is negative definite.
\end{itemize}
\end{definition}
As observed in Section 2, the quadratic form  $Q_U$ associated to the linearized operator at a solution $U$, i.e. to the linear operator
\be \label{linearizedoperator} L_U (V) = - \gD V - J_F (x, U) V 
\ee 
 coincides with the quadratic form corresponding to  the  selfadjoint operator 
 \be \label{symmetriclinearizedoperator} \tilde {L}_U (V) = - \gD V - \frac 12 \left (J_F (x, U) + J_F ^{\text{t}} (x, U)  \right )  V 
\ee 
where $J_F ^{\text{t}}$ is the transpose of the matrix $J_F $.\par 
 Hence  if $\gl _k $ and $W^k$ denote the \emph{symmetric} eigenvalues and eigenfunctions of $L_U$, i.e. $W^k$ satisfy  
 $$ 
  \begin{cases} - \gD W^k + C  W^k =   \gl _k  W^k &  \text{ in }  \gO \\
 W^k =0 & \text{ on } \de \gO \; , 
 \end{cases}
$$
where $C= c_{ij}(x)$,  $c_{ij}(x)= - \frac 12 \left [ \frac {\de f_i}{\de u_j}(x,U(x)) + \frac {\de f_j}{\de u_i}(x,U(x))\right ]$ ,  \\
 as in the scalar case we can prove the following

\begin{proposition} Let $\gO $ be a bounded domain in $\RN$. Then the Morse index of a solution $U$ to \eqref{semilinear system} equals the  number of negative symmetric eigenvalues of the linearized operator $L_U$.
\end{proposition}

\begin{proof} 
 Let us denote by $\mu (U)$ the Morse index as previously defined, and by $m(U)$ the number of negative symmetric eigenvalues.\\
  If $Q_U$ is negative definite on  a $m$-dimensional subspace of $ C_c^1 (\gO;\Rm)$, by   Proposition \ref{varformautov} iii)  the $m$-th symmetric eigenvalue $\gl _m $ is negative, so that $m (U)  \geq \gm (U)$. \\
  On the other hand if there are $m$ negative eigenvalues of the symmetric operator $\tilde {L}_U$ in $\gO $, by the continuity of the eigenvalues there exists a subdomain $\gO ' \subset \gO $ where there are $m$ negative eigenvalues and  corresponding orthogonal eigenfunctions $W^1, \dots , W^m$ which by trivial extension can be considered as functions  with compact support in $\gO $. Regularizing  these functions  we get that the quadratic form
 $Q_U$ is negative definite on a subspace of $ C_c^1 (\gO;\Rm)$ spanned by $m$ linear independent functions, so that $\mu (U)  \geq m(U)$.
\end{proof}
 
\medskip
\subsection{Preliminary results}
Let $e\in S^{N-1}$ be a direction, i.e. $e \in \RN $, $|e|=1$,  and let us define the hyperplane $T(e)$ and the  ''cap'' $\gO (e) $ as
$$ T(e)  = \{ x \in \RN : x \cdot e=0\}\; , \quad  \gO (e)= \{ x \in \gO : x \cdot e >0\}  $$
 Moreover if $x \in \gO $ let us denote by $\gs _{e}(x)$ the reflection of $x$ through the hyperplane $T(e)$ and by $U^{\gs _{e}}$ the function $U \circ \gs _{e} $ .\\
 
 \begin{lemma}\label{lemma1} Assume that $U$ is any solution of \eqref{modprob} and that 
  the  hypotheses i)--iii)  of Theorem \ref{fconvessa} hold.
  Then, for any direction $e\in S^{N-1}$,  the function   $W^{e}=U-U^{\gs _{e}} = (w_1, \dots , w_m)$  satisfies  in $\gO (e)$ (and in $\gO $)   a linear system 
\be \label{EquazDifferenza}
\begin{cases}  - \gD W + B_e(x) W & =0 \quad \text{ in }  \gO (e) \\
W& =0 \quad \text{ on }  \de \gO (e)
\end{cases}
\ee 
i.e.,  if $B_e =(b_{ij})_{i,j=1}^m$,
$$
\begin{cases}
 - \gD w_1 +  b_{11} w_1+  \dots   +b_{1m} w_m =0  & \text{ in } \gO  (e) \\
  \dots  \dots & \dots \\ 
 - \gD w_m +  b_{m1} w_1+  \dots   +b_{mm} w_m =0  & \text{ in } \gO (e) \\
  w_1=  \dots  = w_m   =0  & \text {  on } \de \gO (e) \, ,
\end{cases}
$$
such that for any  $i,j=1,\dots ,m$  and $ x \in \gO $  the following hold:  
\be \label{monotoniadiagonali} b_{ii} (x) \geq  -  \frac {\de f _i} {\de u_i} (|x|,U(x))\quad \text{if} \quad  u_i (x) \geq u_i^{\gs _e}(x)\; ,  
\ee  
\be  \label{monotonianondiagonali} b_{ij} (x) \geq  -  \frac {\de f _i} {\de u_j} (|x|,U(x))\quad \text{if} \quad  u_i (x) \geq u_i^{\gs _e}(x)\; , \;  u_j (x) \geq u_j^{\gs _e}(x)
\ee  
 while if $u_i (x) = u_i^{\gs _e}(x)$, $u_j (x) = u_j^{\gs _e}(x)$ then $b_{ij} (x) =  -  \frac {\de f _i} {\de u_j} (|x|,U(x))$.\par 
 Moreover the system \eqref{EquazDifferenza} is fully coupled along $W^{e}$ in $\gO (e) $.
\end{lemma}
\begin{proof} From the equation $- \gD U = F (|x|,U(x)) $ we deduce that the reflected function $U_{\gs _{e}} $ satisfies the equation
$- \gD U^{\gs _{e}}  = F (|x|,U^{\gs _{e}}(x) ) $ and hence  the difference $W^{e}=U-U^{\gs _{e}} $  satisfies  \\
$- \gD W^{e}  = F(|x|, U)-F(|x|, U^{\gs _{e}})$.\par
We can write the first equation of this system as
\begin{multline} \notag  - \Delta w_1=\left [ f_1(|x|, u_1,u_2, \dots , u_m)- f_1(|x|, u_1^{\gs _{e}}, u_2 ,\dots , u_m) \right ]+ \\
\left [ f_1(|x|, u_1^{\gs _{e}}, u_2 ,\dots , u_m) - f_1(|x|, u_1^{\gs _{e}},  u_2 ^{\gs _{e}} ,\dots , u_m) \right ]+   \\
  \dots \left [f_1(|x|, u_1 ^{\gs _{e}},u_2 ^{\gs _{e}}, \dots , u_{m-1} ^{\gs _{e}}, u_m) - f_1(|x|, u_1 ^{\gs _{e}},u_2 ^{\gs _{e}}, \dots , u_m ^{\gs _{e}}) \right ]
\end{multline}
Then we define, using iii): 
$$ b_{11}(x)= -\int _0^1 \sum _{k\neq 1} \frac {\de g _{1k}} {\de u_1} \left [|x|, tu_1(x)+ (1-t)u_1(x^{\gs _e}),u_k(x)  \right ]dt 
 \;  ,
$$ 
$$
b_{12}(x)
= -  \int _0^1  \frac {\de g _{12}} {\de u_2} \left [|x|, u_1(x^{\gs _e}) ,t u_2(x) + (1-t) u_2 (x^{\gs _e})  \right ] dt \; ,  
$$
\dots 
$$   b_{1m}(x)
= -  \int _0^1  \frac {\de g _{1m}} {\de u_m} \left [|x|, u_1(x^{\gs _e}), t u_m(x) + (1-t) u_m (x^{\gs _e}) \right ]dt  
$$
Hence we can write the equation as
$$ - \Delta w_1 + b_{11}(x)(u_1-u_1 ^{\gs _{e}})+ \dots + b_{1m}(x)(u_m-u_m^{\gs _{e}}) =0
$$
Proceeding analogously for the other equations and letting $i$ playing the same role of $1$ in the first equation, we get that  $W^{e}$ satisfies the linear system \eqref{EquazDifferenza}
 whith $B_e= (b_{ij})_{i,j=1}^m $ defined by  
 \be  \label{coefficientidiagonali} b_{ii}(x)=  -\int _0^1 \sum _{k\neq i} \frac {\de g _{ik}} {\de u_i} \left [|x|, tu_i(x)+ (1-t)u_i(x^{\gs _e}),u_k(x)  \right ]dt 
 \ee 
   while if $j \neq i$  
 \be  \label{coefficientinondiagonali} b_{ij}(x)= -  \int _0^1  \frac {\de g _{ij}} {\de u_j} \left [|x|,u_i^{\gs _e}  ,t u_j(x) + (1-t) u_j ^{\gs _e} (x)  \right ] dt
 \ee 
 By ii) and  iii), using Remark \ref{significatocondizioni},  we get  \eqref{monotoniadiagonali}, \eqref{monotonianondiagonali}.\par
  Obviously  if $u_i (x) = u_i(x^{\gs _e})$, $u_j (x) = u_j(x^{\gs _e})$ then $b_{ij} (x) =  -  \frac {\de f _i} {\de u_j} (|x|,U(x))$.\par   
Since the system \eqref{modprob} is weakly coupled in $\gO (e) $,  $ \frac {\de f _i} {\de u_j} = \frac {\de g _{ij}} {\de u_{j}} \geq 0  $ for any $i \neq j$, so that $b_{ij} \leq 0 $ and the system \eqref{EquazDifferenza} is weakly coupled in $\gO (e) $ as well.\\
The system  \eqref{EquazDifferenza} is fully coupled in $\gO (e) $ because if $I,J$ are as in Definition \ref{semilinearicoupled}, 
there exist   $i_0 \in I$, $ j_0 \in J$, $\overline{x} \in \gO (-e) $,  such that   
$ \frac {\de f _{i_0}} {\de u_{j_0}}(|\overline{x}|,U(\overline{x}))= \frac {\de g _{i_0j_0}} {\de u_{j_0}}(|\overline{x}|,u_{i_0}(\overline {x}), u_{j_0}(\overline {x})) >0$, since by hypothesis i) the system \eqref{modprob}  is fully coupled along $U$ in $\gO (-e)$.\\
Then by the nonnegativity and continuity of $ \frac {\de g _{i_0j_0}} {\de u_{j_0}} $ 
also 
$$ b_{i_0 j_0}(\overline{x} ^{\gs _e}) = -  \int _0^1  \frac {\de g _{i_0j_0}} {\de u_{j_0}} \left [|\overline {x}|,u_{i_0}(\overline {x})  ,t u_{j_0}  (\overline {x}^{\gs _e})   + (1-t) u_{j_0}(\overline {x})  \right ] dt $$ 
is negative, with $\overline{x} ^{\gs _e} \in \gO (e)$ .
 \end{proof}

\begin{lemma}\label{lemma2} Let $U=(u_1, \dots , u_m)$ be a solution of \eqref{modprob} and assume that the hypothesis i) of Theorem \ref{fconvessa} holds.
 If for every $e\in S^{N-1}$ we have either $U \geq U^{\gs _{e}}$ in $\gO (e)$ or $U \leq U^{\gs _{e}}$ in $\gO (e)$, then $U$ is foliated Schwarz symmetric.
\end{lemma}
\begin{proof} Let us start by proving that each component $u_i$ of the solution is foliated Schwarz symmetric with respect to a vector $p_i \in \RN$, $i=1, \dots , m$. To do this we argue as in Lemma 2.4 in \cite{GPW} .\\
Let $p_i \in S^{N-1}$ be such that for some $r>0$   $S_r=\{  x \in \RN : |x|=r\} \subset \gO $ and  $u_i(r p_i)= \max _{S_r} u_i $.  \\
We define $T_{p_i}^{+}= \{ e \in S^{N-1}: e \cdot p_i >0 \}$ and $T_{p_i}^{-}= \{ e \in S^{N-1}: e \cdot p_i <0 \}$. To prove that $u_i$ is foliated Schwarz symmetric with respect to $p_i$ it suffices to show that 
\be \label{T+} u_i (x) \geq u_i (\gs _{e}(x)) \quad \text{ for all } x \in \gO (e)
\ee
whenever $e \in T_{p_i}^{+}$. \\
Indeed this would imply the reverse inequality for $e \in T_{p_i}^{-}$, so that $u_i$ would be axially symmetric about the axis with direction $p_i$, and the angular monotonicity would follow also by the same inequalities.\\
So for fixed $e \in T_{p_i}^{+}$ we consider the difference $ w_i = u_i - u_i^{\gs _e}$,  and  
by assumption either $w_i \geq 0 $, or   $w_i  \leq  0 $. 
In the first case  \eqref{T+}, is satisfied.\\
 In the second case, since the system \eqref{EquazDifferenza}   is cooperative, we have that $b_{ij} \leq 0 $ if $i \neq j$  so that, since $w_j \leq 0 $, 
$$ - \gD w_i + b_{ii} w_i = -\sum _{i \neq j} b_{ij}w_j
 \leq 0 \; , \quad w_i \leq 0 \quad \text{ in } \gO (e)
 $$
 Then, by the (scalar)  Strong Maximum Principle, we get that either $w_i <0$ or $w_i \equiv 0$ in $\gO (e)$, and the first case is not possible, since $rp_i \in \gO (e) $
 and $w_i (rp_i) \geq 0 $ because $u_i (rp_i)$ is the maximum of $u_i$ on $S_r$.  
 Hence $w_i \equiv 0 $ in $\gO (e)$ and \eqref{T+} is satisfied.\\
 At this point, to prove that the solution $U$ is foliated Schwarz symmetric we only need to prove that the vectors $p_i$ are the same vector for all $i=1, \dots , m$.
 To this aim consider the vector $p_1$ and take any hyperplane $T(e)= \{ x \in \RN : x \cdot e =0 \}$ passing through the axis with direction $p_1$.
 The difference $W^{e}=U-U^{\gs _{e}} = (w_1, \dots , w_m)$  satisfies  in $\gO (e)$ the linear system 
\eqref{EquazDifferenza}, which is fully coupled along $W^{e}$ in $\gO (e)$.  Moreover either $U \geq U^{\gs _{e}}$ in $\gO (e)$ or $U \leq U^{\gs _{e}}$ in $\gO (e)$ by hypothesis, and $w_1 \equiv 0 $ in $\gO (e)$. 
So  by   Theorem \ref{SMP} necessarily we have that $w_j  \equiv 0 $ in $\gO (e)$ for all $j=1,\dots ,m $, i.e. all vectors $p_j $ coincide with $p_1$. 
 \end{proof}
 
\begin{lemma}\label{lemma3} Let $U= (u_1, \dots , u_m)$  be a solution of \eqref{modprob} and assume that the hypothesis i) of Theorem \ref{fconvessa} holds.\\
Suppose that there exists a direction $e$ such that $U$ is symmetric with respect to $T(e)$ and 
the principal eigenvalue $\tilde {\gl } _1 (\gO (e)) $ of the linearized operator $L_U (V) = - \gD V - J_F (x, U) V $ in $\gO (e)$ is nonnegative.
Then $U$ is foliated Schwarz symmetric. 
\end{lemma}

\begin{proof}  We adapt to the case of systems the proof of Proposition 2.3 in \cite{PW}.\\
After a rotation,  we may assume that $e = e_2= (0,1,\dots 0)$ so that $T(e)=\{ x  \in \RN : x_2 =0 \}$. To prove the assertion we will use  Lemma \ref{lemma2}, so we consider a direction $e' \in S^{N-1}$, $e' \neq e$. After another orthogonal transformation which leaves $e_2$ and $T(e_2) $ invariant, we may assume that 
$e' = (\cos \theta _0 , \sin \theta _0,0, \dots ,0)$, for some $\theta _0 \in (- \frac {\pi}2, \frac {\pi}2)$. Now we choose new coordinates, replacing $(x_1, x_2)$ by polar coordinates $(r, \theta ) $, $x_1= r \cos \theta $,  $x_2= r  \sin \theta $,  and leaving $\tilde {x}= (x_3, \dots x_N) $ unchanged.\\
The derivative $U_{\theta} $ of $U$ with respect to $\theta $, extended in the origin with $U_{\theta}(0)=0$ if $\gO $ is a ball, satisfies the linearized system, i.e.
\be \label{sistema  linearizzato} - \gD U _{\theta} - J_F (|x|, U) U_{\theta} =0 \quad \text{ in } \gO (e_2)
\ee 
Moreover, by the symmetry of $U$ with respect to the hyperplane $T(e_2) $ we have that $U_{\theta}$ is antisymmetric with respect to $T(e_2) $ and therefore vanishes on $T(e_2) $ and since it vanishes on $\de \gO $, it vanishes on $\de \gO(e_2)$ as well.\\
Thus if $\tilde {\gl _1} (\gO (e)) >0$, since the maximum principle holds, we get that $U_{\theta } \equiv 0$ in $\gO (e)$.\\
 If instead $\tilde {\gl _1} (\gO (e)) =0$ and $U_{\theta} \not \equiv 0 $, by the simplicity of the principal eigenvalue we get that $U_{\theta}$ is a principal eigenfunction and hence  it is positive (or negative) in $\gO (e)$.\\
   In any case $U_{\theta}$ does not change sign in $\gO (e)$, and arguing exactly as in Proposition 2.3 of \cite{PW} we get that for any direction $e$ either 
$U \geq U^{\gs _e}$ or  $U \leq U^{\gs _e}$. Thus, by Lemma \ref{lemma2}, $U$ is foliated Schwarz symmetric.\\
 \end{proof}
 
 \begin{lemma}\label{lemma4}  Assume that $U$ is a solution of \eqref{modprob}  and that the hypotheses i)--iii) of Theorem \ref{fconvessa} hold. 
Then   
$$Q_U\left ( (W^{e})^{+}; \gO (e)  \right )=Q_U\left ( (U-U^{\gs _{e}})^{+} ; \gO (e) \right ) \leq 0 \quad \quad  \forall \; e \in S^{N-1}
$$
 where $Q_U$ is the quadratic form defined in \eqref{formaquadraticalinearizzato}.   

 \end{lemma}
\begin{proof}
 As we saw in Lemma \ref{lemma1} the difference $W^{e}=U-U^{\gs _{e}} $ satisfies the linear system \eqref{EquazDifferenza}. 
Let us multiply the $i$-th equation of this system by $(u_i- u_i^{\gs _{e}} )^{+}$ and integrate. By Lemma \ref{lemma1} we get, since   $b_{ik} \leq 0$ if $i \neq k $ and  in the domain of integration where $b_{ik}$ appears
$u_i(x) \geq u_i^{\gs _e}(x)$, $u_k(x) \geq u_k^{\gs _e}(x)$ (so that   $b_{ik} (x) \geq  -  \frac {\de f _i} {\de u_k} (|x|,U(x))$) :
 \begin{multline} 0= 
 \int _{\gO}  \left ( | \nabla (u_i- u_i^{\gs _{e}} )^{+}|^2 
+ b_{ii}(x) [(u_i- u_i^{\gs _{e}} )^{+}]^2  \right )+  \\
  \sum _{k \neq i} \int _{\gO} \left (  b_{ik}(x) [(u_i- u_i^{\gs _{e}} )^{+}]\left [ (u_k - u_k^{\gs _{e}})^{+}  -(u_k - u_k^{\gs _{e}})^{-}   \right ]  \right )\geq \\
 \int _{\gO} \left (  |\nabla (u_i- u_i^{\gs _{e}} )^{+}|^2 
+ b_{ii}(x) [(u_i- u_i^{\gs _{e}} )^{+}]^2\right )  +\\
  \sum _{k \neq i}  \int _{\gO} \left (  b_{ik}(x) [(u_i- u_i^{\gs _{e}} )^{+}]\left [ (u_k - u_k^{\gs _{e}})^{+}     \right ] \right ) \geq \\
  \int _{\gO} \left (  |\nabla (u_i- u_i^{\gs _{e}} )^{+}|^2 
-     \frac {\de  f_i}{\de u_i}(|x|, u_1, u_2, \dots , u_m) [(u_i- u_i^{\gs _{e}} )^{+}]^2 \right )  \\
-    \sum _{k \neq i}  \int _{\gO} \left ( 
 \frac {\de  f_i}{\de u_k} (|x|, u_1, u_2, \dots , u_m) (u_i- u_i^{\gs _{e}} )^{+}(u_k - u_k^{\gs _{e}})^{+}  \right ) 
\end{multline}
 for any $i=1,\dots ,m$. Hence  we obtain that 
\be 0 \geq 
Q_U \left ((u_1-u_1^{\gs _{e}})^{+}, \dots ,  (u_m-u_m^{\gs _{e}})^{+} ; \gO (e) \right )
\ee
 \end{proof}

\begin{lemma}\label{lemma5}  Suppose that $U$ is a solution of \eqref{modprob} with Morse index $\mu (U) \leq N$ and assume that the hypothesis i) of Theorem \ref{fconvessa} holds.
Then there exists a direction $e \in S^{N-1}$ such that 
$$
 Q_U (\Psi ; \gO (e) )  \geq 0
  \notag
$$
for any $\Psi \in C_c^1 (\gO (e);\Rm)$ .\\
Equivalently  
the first symmetric eigenvalue $\gl _1^{\text{s}} (L_u, \gO (e))$ of the linearized operator $L_U (V) = - \gD V - J_F (x, U) V $ in $\gO (e)$ is nonnegative (and hence also the principal eigenvalue $\tilde {\gl }_1 (L_U ,\gO (e) )$ is  nonnegative).
\end{lemma}
\begin{proof} The proof is immediate if $\mu (U) \leq 1$. In fact in this case for any direction $e$ at least one amongst $\gl _1^{\text{s}} (L_u, \gO (e))$  and 
$\gl _1^{\text{s}} (L_u, \gO (-e))$ must be nonnegative, otherwise taking the  corresponding first eigenfunctions we would obtain a $2$-dimensional subspace of  $C_c^1(\gO ; \Rm)$ where the quadratic form is negative definite.  \\
So let us assume that  $2\leq j= \mu (U) \leq N $ and let $\Phi _1, \dots , \Phi _j$ be mutually orthogonal eigenfunctions corresponding to the negative symmetric eigenvalues $\gl _1^{\text{s}} (L_u, \gO )$, \dots , $\gl _j^{\text{s}} (L_u, \gO )$.\\
 For any $e\in S^{N-1}$ let $\Phi _{e}$ be the first positive $L^2$ normalized eigenfunction  of the symmetric system associated to the linearized operator 
 $L_U (V) $ in $\gO (e)$. We observe that $\Phi _{e}$ is uniquely determined by the assumption i).
  Define
 \be  \Psi _{e} (x) =  
 \begin{cases} \left (  \frac {  ( \Phi _{-e} \,, \, \Phi _1 )_{\ldue (\gO (-e))} } { ( \Phi _{e}  \,, \, \Phi _1 )_{\ldue (\gO (e))}}  \right )^{\frac 12} \Phi _{e} (x)   \quad &\text{ if } x \in \gO (e)    \\
 - \,   \left (  \frac {  ( \Phi _{-e}  \,, \, \Phi _1 )_{\ldue (\gO (-e))} } { ( \Phi _{e}  \,, \, \Phi _1 )_{\ldue (\gO (e))}}  \right )^{\frac 12} \Phi _{-e} (x)                       \quad &\text{ if } x \in \gO (-e)
\end{cases}
 \ee  
 The mapping $e \mapsto \Psi _e $ is a continuous odd function from $S^{N-1}$ to $\huno $  and, by construction,  
 $ ( \Psi _e \,, \, \Phi _1 )_{\ldue (\gO )} =0$. \\
 Therefore the mapping $h: S^{N-1}\to \R ^{j-1} $ defined by
 $$ h(e)= \left ( ( \Psi _e \,, \, \Phi _2 )_{\ldue (\gO )},   \dots , ( \Psi _e \,, \, \Phi _j )_{\ldue (\gO )} \right )
 $$
 is an odd continuous mapping, and since $j-1 < N $,  by the Borsuk-Ulam Theorem it must  have a zero.
 This means that there exists a direction $e\in S^{N-1}$ such that
   $\Psi _e $ is orthogonal to all the eigenfunctions $\Phi _1, \dots , \Phi _j$. Since $ \mu (U)=j $  this implies that 
 $Q_U(\Psi _e ; \gO )  \geq 0 $, which in turn implies that either $Q_U(\Phi _e ; \gO (e) ) \geq 0 $ or $Q_U(\Phi _{-e}; \gO (-e) ) \geq 0 $, i.e. either
 $\gl _1^{\text{s}} (L_u, \gO (e))$ or $\gl _1^{\text{s}} (L_u, \gO (-e))$ is nonnegative, so the assertion is proved.
\end{proof}
\subsection{Proof of the  symmetry results}
\begin{proof}[Proof of Theorem \ref{fconvessa}]
By Lemma \ref{lemma5} there exists a direction $e$ such that the first symmetric eigenvalue $\gl _1^{\text{s}} (L_u, \gO (e))$ of the linearized operator 
  is nonnegative, so that the principal eigenvalue $\tilde {\gl}_1 (\gO (e))$ is nonnegative as well.
 Moreover by Lemma \ref{lemma4} we have that $Q_U\left ( (U-U^{\gs _{e}})^{+}  \right ) \leq 0 $, so that either $ (U-U^{\gs _{e}})^{+} \equiv 0 $, or $\gl _1^{\text{s}} (L_u, \gO (e))=0 $ and  $ (U-U^{\gs _{e}})^{+}$ is the positive first symmetric eigenfunction in $\gO (e)$. In any case we have that either $U\leq U^{\gs _{e}}$ or 
 $U\geq U^{\gs _{e}}$ in $\gO (e)$ holds.\\
If $U \equiv  U^{\gs _{e}} $ in $\gO (e)$ then by Lemma \ref{lemma3}, since the principal eigenvalue $\tilde {\gl}_1 (\gO (e))$ is nonnegative,  we immediately obtain that the solution $U$ is foliated Schwarz symmetric.
 If this is not the case and e.g. $U \geq U^{\gs _{e}} $ then,   by Lemma \ref{lemma1} 
  we have that $U > U^{\gs _{e}} $  in $\gO (e)$. \\
 We now apply, as in \cite{PW} and \cite{GPW}, the ''rotating plane method'', which is an adaptation of the Moving Plane method as developed in \cite {BN} and obtain a different direction $e'$ such that $U$ is symmetric with respect to $T(e')$ and 
the principal eigenvalue $ \tilde {\gl} _1 (L_u, \gO (e'))$ of the linearized operator  in $\gO (e')$ is nonnegative.
Then by Lemma \ref{lemma3} we get that  $U$ is foliated Schwarz symmetric. \\
More precisely, without loss of generality we suppose that $e=(0,0,\dots ,1)$ and  for $\theta \geq 0 $ we set $e_{\theta}=(\sin \theta ,0,\dots ,\cos \theta) $, so that $e_0=e$, and
$\gO _{\theta}= \gO (e_{\theta})$, $U^{\theta}= U^{\gs _{e_{\theta}}}$, $W^{\theta}=U-U^{\gs _{e_{\theta}}} $, 
Let us define $\theta _0 = \sup \{ \theta \in [0, \pi ) : U > U ^{\theta} \text{ in } \gO _{\theta}  \}$.
Then necessarily $\theta _0 < \pi$, since $(U-U^0)(x) = - (U-U^{\pi})(\gs _{e_{\pi}}(x))$ for any $x \in \gO _0$ (and $\gs _{e_{\pi}}(x)) \in \gO _{\pi}$).\\
Suppose by contradiction that $U \not \equiv U^{\theta _0}$ in $\gO _{\theta _0}$.
Then, by Lemma \ref{lemma1}) applied to    the difference $W^{\theta _0}=(U - U^{\theta _0})$,  we get that $W^{\theta _0}>0$ in $\gO _{\theta _0}$. Taking a compact $K \subset \gO _{\theta _0} $ whith small measure and such that (componentwise) 
$W^{\theta _0}>(\eta, \dots , \eta)$ for some $\eta >0$, for $\theta $ close to $\theta _0 $ we still have that $W^{\theta}>(\frac {\eta }2, \dots , \frac {\eta }2) $ in $K$, while 
$W^{\theta}>0$ in $\gO _{\theta} \setminus K$ by the weak maximum principle in domains with small measure. 
This implies that for $\theta $ greater than and close to $\theta _0 $
the inequality $U > U ^{\theta} \text{ in } \gO _{\theta} $ still holds, contradicting the definition of $\theta _0$. \\
Therefore $U \equiv U^{\theta _0}$ in $\gO _{\theta _0}$.\\
Observe that the difference $ W^{\theta} $ satisfies the linear system \eqref{EquazDifferenza} and does not change sign for any $\theta \in [0, \theta _0)$, which implies by Proposition \ref{principaleigenvalue} that it is the principal eigenfunction for the system \eqref{EquazDifferenza} corresponding to the eigenvalue $\tilde {\gl} _1 =0 $.
As $\theta \to \theta _0 $ the system \eqref{EquazDifferenza} tends to the linearized system, in the sense that the coefficients $b_{ij}(x)$ tend to the derivatives 
$\frac {\de f_i}{\de u_j}(|x|, U(x))$.  Then by continuity  the principal eigenvalue $\tilde {\gl } _1 (\gO (e_{\theta _0})) $ of the linearized operator $L_U (V) = - \gD V - J_F (x, U) V $ in $\gO (e_{\theta _0})$ is zero.
\end{proof}

\begin{proof}[Proof of Theorem \ref{teorema2}]
 Let us  choose $N$ orthogonal directions $e_1$, $e_2 , \dots e_N \in S^{N-1}$  and introduce new (cylinder) coordinates  
$(r, \theta, y_3, \dots ,y_N)$ defined by the relations $x= r [\cos \theta e_1 + \sin \theta e_2] + \sum_{i=3}^N y_i e_i $.   \\ 
Then the angular derivative $U_{\theta} $ of $U$ with respect to $\theta $, extended by zero at  the origin if $\gO $ is a ball, satisfies the linearized system, i.e.
\be \label{sistemalinearizzato}
\begin{cases}
 - \gD U _{\theta} - J_F (|x|, U) U_{\theta} &=0 \quad \text{ in } \gO  \\
 U _{\theta} &=0 \quad \text{ on } \de \gO 
 \end{cases}
\ee 
By the stability assumption we have that  the first symmetric eigenvalue of the linearized system in $\gO $ is nonnegative, so that by iv) of Proposition \ref{principaleigenvalue} also  the principal eigenvalue $\tilde {\gl _1}(\gO )$ is nonnegative.\par
Then by \eqref{sistemalinearizzato} and Proposition \ref{principaleigenvalue} we get that 
   $U_{\theta}$, if does not vanish, must be the first symmetric eigenfunction corresponding to the eigenvalue $0$ and hence it does not change sign in $\gO$. 
   Since $U_{\theta}$ is $2 \gp $- periodic this is impossible and  therefore $U_{\theta}\equiv 0 $. By the arbitrarity of the vectors $e_1$, $e_2$ we conclude that $U$ is radial.
 \end{proof}

\begin{proof}[Proof of  Corollary \ref{corollario1}]As remarked at the beginning of the proof of Lemma \ref{lemma5},
if  $U$ is a Morse index one solution for any direction $e$ either   $\gl _1^{\text{s}} (L_u, \gO (e))$ or $\gl _1^{\text{s}} (L_u, \gO (-e))$ must be nonnegative.
By the  proof of Theorem \ref{fconvessa}  we can find a direction $e$ such that $U$ is symmetric with respect to the hyperplane $T(e)$ and the principal eigenvalue 
$ \tilde {\gl} _1 (\gO (e)) = \tilde {\gl} _1 (\gO (-e))  \geq 0 $.
On the other hand by the symmetry $ \gl _1 ^{\text{s}} (\gO (e)) = \gl _1 ^{\text{s}} (\gO (-e))  \geq 0 $.
If $ \tilde {\gl} _1 (\gO (e)) >0 $ then, by the maximum principle,  the derivatives $U_{\theta} $ as defined in the proof of Lemma \ref{lemma3}, must vanish , and hence 
$U$ is radial. So if  $U_{\theta} \not \equiv 0 $ necessarily   $ \tilde {\gl} _1 (\gO (e)) = \gl _1 ^{\text{s}} (\gO (e))  =0 $ and by v) of Proposition \ref{principaleigenvalue} 
$U_{\theta}$ is the first eigenfunction of the simmetrized system, as well as a solution of \eqref{sistemalinearizzato}. So  we get that 
$J_F (|x|, U) U_{\theta}= \frac 12 \left (J_F (|x|, U) + J_F ^t (|x|, U) \right ) U_{\theta}$, i.e. 
  \eqref{superfullycoupling1} and if $m=2$, since $U_{\theta}$ is positive, we get    \eqref{superfullycoupling2}.
\end{proof}

 \section{  \textbf{Examples and comments} }
 A first type of elliptic systems that could be considered are those of ''gradient type'' (see \cite{deF2}), i.e. systems of the type \eqref{modprob}
 where $f_j(|x|,U)= \frac {\de g} {\de u_j}(|x|,U)$ for some scalar function $g \in C^{2, \ga }([0,+\infty ) \times \Rm)$.\\
 In this case the solutions correspond to critical points of the functional
 $$ \Phi (u) = \frac 12 \int _{\gO } |\nabla U|^2 \, dx - \int _{\gO } g(|x|,U) \, dx
 $$
 in $\huno (\gO )$ and the linearized operator \eqref{linearizedoperator} coincides with the second derivative of $\Phi $.\par 
 Thus standard variational methods apply which often give solutions of finite (linearized) Morse index, as, for example, in the case  when the Mountain Pass Theorem can be used. So, if the hypotheses of Theorem \ref{fconvessa} are satisfied, our symmetry results can be applied.\par
 Many systems of this type have been studied (see  \cite{deF2} and the references therein). An example is the nonlinear Schr\"odinger system
  \be \label{schrodingersystem} 
 \begin{cases}   - \Delta u_1 + u_1 &= |u_1|^{2q -2}u_1  + b |u_2|^q |u_1|^{q-2}u_1 \text{ in } \gO  \\
 - \Delta u_2 + \go ^2 u_2 &= |u_2|^{2q -2}u_2  + b |u_1|^q |u_2|^{q-2}u_1 \text{ in } \gO
 \end{cases}
 \ee
 where $b>0$, $\gO$ is either a bounded domain or the whole $\RN$, $N \geq 2$, $q>1 $ if $N=2$ and $1<q< \frac N{N-2}$ if $N \geq 3$.
 If $\gO \neq \RN $ the Dirichlet boundary conditions are imposed
 \be \label{schrodingersystembordo} u_1=u_2=0 \text{ on }  \partial \gO  
 \ee
 This type of system has been studied in several recent papers (see \cite{AC}, \cite{MMP} and the references therein).\\
 It is easy to see that the system is of gradient type and the solutions of \eqref{schrodingersystem}, \eqref{schrodingersystembordo} are  critical points of the functional
 \be I(U)=I(u_1, u_2)= \frac 12 \int _{\gO } |\nabla U| ^2 \, dx - \frac 1{2q} \int _{\gO }( |u_1|^{2q} + |u_2|^{2q}) \, dx  - \frac bq  \int _{\gO } |u_1 u_2 |^q \, dx
 \ee
 in the space $\huno (\gO )$.\par
 Thus the linearized operator \eqref{linearizedoperator} at a solution $U$ corresponds  to the second derivative of $I$ in $\huno (\gO )$ and hence the corresponding linear system is symmetric.\par
 Moreover, as observed in \cite{MMP}, the system \eqref{schrodingersystem} is cooperative, and fully coupled in $\gO (e)$, for any $e \in S^{N-1}$, along every purely vector solution $U=(u_1,u_2)$, i.e. such that  $u_1$ and  $u_2$ are both not identically zero. In particular  the system is fully coupled along  positive solutions in any $\gO (e)$.\\
 Then Theorem \ref{fconvessa} could be applied to any purely vector solution with Morse index less than or equal to $N$, in an annulus or in a ball. In particular we can consider solutions obtained by the Mountain Pass Theorem which have Morse index  not bigger than one which are purely vector positive solutions for suitable values of $b$ and $q \geq 2$, as could be obtained by the same proof of \cite{MMP} for the case $\gO = \RN$.\par 
 
 \medskip
 A second type of interesting systems  of two equations are the so called ''Hamiltonian type'' systems (see  \cite{deF2} and the references therein).
 More precisely we consider the system
 \be \label{hamiltoniansystem}
\begin{cases}  
 - \gD u_1  = f_1(|x|,u_1, u_2) & \text{ in } \gO  \\
 - \gD u_2 = f_2(|x|,u_1, u_2)   & \text{ in } \gO \\
 u_1=  u_2  =0 & \text { on } \de \gO 
\end{cases}
\ee
  with
  \be \label{hamiltonianconditions}
  f_1 (|x|,u_1, u_2)=  \frac {\de H} {\de u_2}(|x|,u_1,u_2) \quad , \quad f_2 (|x|,u_1, u_2)=  \frac {\de H} {\de u_1}(|x|,u_1,u_2)
  \ee
  for some scalar function $H \in C^{2, \ga }([0,+\infty ) \times \R ^2)$.
  These systems  can be studied by considering the associated functional
   \be J(U)=I(u_1, u_2)= \frac 12 \int _{\gO }  \nabla u_1 \cdot \nabla u_2  \, dx - \int _{\gO } H(|x|,u_1, u_2) \, dx
 \ee
 either in  $ \huno (\gO) $ or in other suitable Sobolev spaces (see \cite{deF2}). \par
 It is easy to see that the linearized operator defined in \eqref{linearizedoperator} does not correspond to the second derivative of the functional $J$, which is strongly indefinite. Nevertheless  solutions of \eqref{hamiltoniansystem} can have finite linearized Morse index  as we show with a few simple examples.\par
 \medskip
 Let us consider the following system:  
 \be \label{esponenziale}
 \begin{cases}
 - \Delta u_1 = \gl e^{u_2} & \text{ in } \gO \\
 - \Delta u_2 = \gm e^{u_1} & \text{ in } \gO \\
 u_1=  u_2  =0 & \text { on } \de \gO 
 \end{cases}
 \ee
 where $\gl ,  \gm \in \R $. \\
  If $\gl $, $\gm >0 $ then the system is fully coupled along every solution in any $\gO (e)$, $e \in S^{N-1}$, and the hypotheses of Theorem \ref{fconvessa} are satisfied. \\
  Let us consider the function
 $G: \R^2 \times \left ( C^{2, \ga }(\gO ) \right ) ^2 \to  \left ( C^{0, \ga }(\gO ) \right ) ^2$ defined by 
 $ G((\gl , \gm ), (u_1,u_2)) = ( - \Delta u_1 -  \gl e^{u_2}, - \Delta u_2 -  \gl e^{u_1} )$. \\
 We have that $ G((0 , 0 ), (0,0))=(0,0) $,   
  $\frac {\de G} {\de (u_1,u_2)}((0 , 0 ), (0,0))(\phi , \psi)= ( - \Delta \phi ,  - \Delta \psi )  $, and the first (symmetric) eigenvalue of this operator is strictly positive, so that the operator is invertible. \\
  By the implicit function theorem for small $\gl , \gm$ there is a unique nontrivial solution 
  $\left (u_1(\gl , \gm ) , u_2( \gl , \gm ) \right )$ close to the trivial solution of the system \eqref{esponenziale} corresponding to $\gl = \gm =0$. Moreover the solution is positive by the maximum principle and it is linearized stable, so that it is radial.  
  Indeed the  first (symmetric) eigenvalue of the linearized operator corresponding to $\gl = \gm =0$, $ u=v=0$ is strictly positive, and by continuity it is positive for small $\gl , \gm$.\par
  The same happens substituting the exponential with other nonlinearities $f(u_2)$, $g(u_1)$    with nonnegative derivative in a neighborhood of $0$ and such that
  $f(0), g(0) >0$.\par
  \medskip
 
We now consider, as another example, the  system  
 \be \label{potenza}
 \begin{cases}
 - \Delta u_1 &= u_2^{p }  \text{ in }  \gO  \\
 - \Delta u_2 &=  u_1^{q} \text{ in }  \gO  \\
 u_1, u_2 &>0 \text{ in } \gO \\
  u_1=u_2&=0 \text{ on } \de \gO 
 \end{cases}
 \ee
 where $1< p, q < \frac {N+2} {N-2}$.
 The idea is to proceed as in the previous example starting from the case $p=q$ and the solution $u_1=u_2=z$, where $z$ is a scalar  solution of the equation
 \be \label{equazionescalare} 
  \begin{cases}
 - \Delta z &= z^{p }  \text{ in }  \gO  \\
   z&>0 \text{ in }  \gO  \\
 z&=0 \text{ on } \de \gO 

 \end{cases}
 \ee
 which is nondegenerate. \par
 Let us observe that if $p=q$  and $z$  has Morse index equal to the integer $\mu (z)$, then $\mu (z)$ is also the Morse index  of the solution 
 $U=(u_1=u_2)=(z,z)$ of the system \eqref{potenza}. Indeed the linearized equation at $z$ for the equation \eqref{equazionescalare} and the linearized system at $(z,z)$ for the system \eqref{potenza} are respectively
  \be \label{linearizzatoequazionescalare} 
  \begin{cases}
 - \Delta \phi  - p z^{p-1 } \phi & =0  \text{ in }  \gO  \\
  \phi &=0 \text{ on } \de \gO 
 \end{cases}
 \ee
and
 \be \label{linearizzatopotenza}
 \begin{cases}
  - \Delta \phi _1  - p z^{p-1 } \phi _2 & =0  \text{ in }  \gO  \text{ in }  \gO  \\
 - \Delta \phi _2  - p z^{p-1 } \phi _1 & =0  \text{ in }  \gO   \text{ in }  \gO  \\
  \phi _1=\phi _2 &=0 \text{ on } \de \gO 
 \end{cases}
 \ee
This implies that the eigenvalues of these two operators are the same, since if $\phi $ is an eigenfunction for \eqref{linearizzatoequazionescalare} 
corresponding to the eigenvalue $\gl _k $ then taking 
$\phi _1 = \phi _2 = \phi $ we obtain an eigenfunction $(\phi _1 , \phi _2)$ for \eqref{linearizzatopotenza} corresponding to the same eigenvalue, while if 
 $(\phi _1 , \phi _2)$  is an eigenfunction for  \eqref{linearizzatopotenza} corresponding to the eigenvalue $\gl _k $ then $\phi = \phi _1 + \phi _2 $ is an 
 eigenfunction for \eqref{linearizzatoequazionescalare} corresponding to the same eigenvalue.\par 
 So proceeding as in the previous example we can start from a nondegenerate solution of \eqref{equazionescalare} with a fixed exponent 
 $\overline {p} \in (1, \frac {N+2} {N-2} )$   and find a branch of  solutions of \eqref{potenza} corresponding to (possibly different) exponents $p,q$ close to $\overline {p} $.
 For example if we start with a Mountain Pass (positive) solution $z$ in the ball of equation \eqref{equazionescalare} with the exponent $\overline {p} $,  knowing that its Morse index is one and it is nondegenerate, we get a branch of Morse index one radial solutions for  $p,q $ close to $\overline {p} $.\par 
 If $\gO $ is an annulus then a mountain pass solution $z$ of \eqref{equazionescalare} is foliated Schwarz symmetric, as shown in \cite{P}. Then, if not radial, it is obviously degenerate.
 Nevertheless, working in spaces of symmetric functions, we could remove the degeneracy and apply the continuation method described above. \par
Thus, starting from an exponent $\overline{p}$ for which the mountain pass solution $z$ of \eqref{equazionescalare} is not radial, we can construct solutions 
$U=(u_1, u_2)$ of \eqref{potenza} in correspondence of exponents $p,q$ close to $\overline{p}$, with Morse index one.
Then Theorem \ref{fconvessa} applies and, in particular, we get that the coupling condition \eqref{superfullycoupling2} holds, which in this case can be written as
  \be \label{superfullycoupling3}
  pu_2^{p-1} =  qu_1^{q-1}   \quad \text{in } \gO 
\ee 

 Note that more generally, by Corollary \ref{corollario1}, the equality \eqref{superfullycoupling3} must hold for every solution $U=(u_1,u_2)$ of \eqref{potenza} with Morse index one giving  so a sharp condition to be satisfied by the components of a solution of this type.\par
 Let us remark that our results apply also when the nonlinearity depends on $|x|$ (in any way).
 Arguing as before it is not difficult to construct systems having solutions with low Morse index, in particular with Morse index one. \par
  An example could be the "Henon system"
  \be \label{Henon}
 \begin{cases}
 - \Delta u_1 = |x|^{\alpha }u_2^{p } & \text{ in }  \gO  \\
 - \Delta u_2 =  |x|^{\beta } u_1^{q} &\text{ in }  \gO  \\
 u_1, u_2 >0  &\text{ in } \gO \\
  u_1=u_=0 & \text{ on } \de \gO 
 \end{cases}
 \ee
 with $\alpha , \beta >0$, $ p,q >1$. \par 
 Starting again from a solution of the scalar equation of mountain pass type it is possible to construct solutions of \eqref{Henon} with the same Morse index for $p,q$ and $\alpha , \beta $ close to each other.
 Note that, for some values of $\alpha , \beta  $,  $p,q$,  such  solutions would not be radial, even if $\gO $ is a ball, but rather foliated Schwarz symmetric as for the scalar case. \par
 Finally, as in the scalar case, we could consider  nonhomogeneous equations  and provide examples of convex nonlinearities for which there exist solutions of Morse index one which change sign and  to which our results apply.

\end{document}